\documentclass[11pt]{article}

\usepackage{palatino}
\usepackage{amssymb,amsthm,amsmath,amsfonts}
\usepackage{epsfig}

\usepackage{graphicx}
\usepackage{pstricks}
\usepackage{makeidx}
\usepackage{multicol}
\usepackage{amsmath}
\usepackage{amsfonts}
\usepackage{amssymb}
\usepackage{latexsym}
\usepackage{wrapfig}
\usepackage{epsfig}
\usepackage{bm}
\usepackage{lscape}
\usepackage{amsthm}

\begin{document}

\font\grg=eurm10
\def\umu{{\hbox{\grg\char22}}}
\font\grs=eurm10 at 9pt
\def\smu{{\hbox{\grs\char22}}}

\topmargin+30mm

\newcommand{\emkost}{\text{\rm cap}}
\newcommand{\s}{\vspace{0.1cm}}
\newcommand{\ess}{\text{\rm ess }}
\newcommand{\re}{\text{\rm Re}\,}
\newcommand{\rey}{\text{\sc R}\,}
\newcommand{\im}{\text{\rm Im}\,}
\newcommand{\id}{\text{\rm id}}
\newcommand{\res}{\text{\rm Res}\,}
\newcommand{\inte}{\text{\rm int}}
\newcommand{\K}{\text{\bf K}\,}
\newcommand{\Lip}{\text{\rm Lip}\,}
\newcommand{\emk}{\text{\rm cap}\,}
\newcommand{\Fb}{\text{\rm \bf F}}
\newcommand{\Vb}{\text{\rm \bf V}\,}
\newcommand{\Ub}{\text{\rm \bf U}\,}
\newcommand{\Eb}{\text{\rm \bf E}\,}
\newcommand{\Pb}{\text{\rm \bf P}\,}
\newcommand{\xb}{\text{\rm \bf x}\,}
\newcommand{\nb}{\text{\rm \bf n}\,}
\newcommand{\sn}{\text{\rm \bf sn}\,}
\newcommand{\cn}{\text{\rm \bf cn}\,}
\newcommand{\dn}{\text{\rm \bf dn}\,}
\newcommand{\tn}{\text{\rm \bf tn}\,}
\newcommand{\cd}{\text{\rm \bf cd}\,}
\newcommand{\dime}{\text{\rm dim}\,}
\newcommand{\orde}{\text{\rm ord}\,}
\newcommand{\Hol}{\text{\rm Hol}\,}
\newcommand{\Ker}{\text{\rm Ker}\,}
\newcommand{\pv}{\text{\bf p.v.}}
\newcommand{\Beta}{\text{\bf B}}
\newcommand{\Diff}{\text{\rm Diff }}
\newcommand{\Rot}{\text{\rm Rot }}
\newcommand{\Vect}{\text{\rm Vect }}
\newcommand{\Bal}{\text{\rm Bal}\,}
\newcommand{\hull}{\text{\rm conv}\,}
\newcommand{\dist}{\text{\rm dist}\,}
\newcommand{\Ha}{\mathbb{H}}

\newcommand{\eq}{\eqref}
\newcommand{\const}{{\rm{const.}}}
\newcommand{\supp}{{\rm{supp}}}
\newcommand{\cp}{{\rm{cap}}}

\def\D{{\mathbb D}}
\def\R{{\mathbb R}}
\def\C{{\mathbb C}}
\def\P{{\mathbb P}}
\def\Z{{\mathbb Z}}
\def\N{{\mathbb N}}

\def\RE{{\Re}}
\def\MM{{\mathcal{M}}}

\theoremstyle{plain}

\newtheorem{theorem}{Theorem}[section]
\newtheorem{corollary}[theorem]{Corollary}
\newtheorem{proposition}[theorem]{Proposition}
\newtheorem{lemma}[theorem]{Lemma}
\newtheorem{conjecture}[theorem]{Conjecture}

\def\la{\label}
\def\be{\begin{equation}}
\def\beq{\begin{equation}}
\def\eeq{\end{equation}}
\def\ee{\end{equation}}
\def\bea{\begin{eqnarray}}
\def\eea{\end{eqnarray}}
\def\p{\partial}

\newcommand{\ii}{{{i}}}
\newcommand{\dd}{{{d}}}
\newcommand{\sd}{{{d}}}

\newcommand{\Vir}{{\text{\rm Vir}\,}}
\newcommand{\spn}{\text{span}\,} 
\newcommand{\Gr}{\text{Gr}\,}
\newcommand{\kernel}{\text{ker}\,}
\newcommand{\image}{{\text{im}\,}}
\newcommand{\cokernel}{\text{coker}\,}
\newcommand{\virtcard}{{\text{virtcard}\,}}
\newcommand{\virtdim}{{\text{virtdim}\,}}
\newcommand{\ind}{{\text{ind}\,}}
\newcommand{\tb}{\text{\bf t}}


\def\jac{\mathcal{R}}

\def\dif#1{{\frac{d^{#1}}{dx^{#1}}}\,}
\def\I{\mathrm{i}}
\def\Bin#1#2{C_{#1}^{#2}}
\def\scal#1#2{\langle #1, #2\rangle }


\def\Log{{\rm Log\,}}
\def\T{{\mathbb T}}
\def\D{{\mathbb D}}
\def\R{{\mathbb R}}
\def\C{{\mathbb C}}
\def\P{{\mathbb P}}
\def\Z{{\mathbb Z}}
\def\N{{\mathbb N}}
\def\DD{{\mathcal{D}}}
\def\OO{{mathcal{O}}}
\def\Wilde{}
\def\spect#1{\rho(#1)}



\title{Quadrature domains packing}

\author{
Bj\"orn Gustafsson\textsuperscript{1},
Mihai Putinar\textsuperscript{2}}

\footnotetext[1]
{Department of Mathematics, KTH, 100 44, Stockholm, Sweden.
Email: \tt{gbjorn@kth.se}}
\footnotetext[2]
{Department of Mathematics, University of California at Santa Barbara, Santa Barbara, CA 93106-3080, USA.
Email: {\tt{mputinar@math.ucsb.edu}}, School of Mathematics Statistics and Physics, Newcastle University, Newcastle upon Tyne,
NE1 7RU, U. K. Email: {\tt{mihai.putinar@ncl.ac.uk}}}

\date{\today}
\maketitle

{To Ed Saff, with best wishes on the occasion of his eightieth birthday}
\bigskip

\begin{abstract} Given a finite family of compact subsets of the complex plane we propose a certificate of mutual non-overlapping with respect to area measure.
The criterion is stated as a couple of positivity conditions imposed on a four argument analytic/anti-analytic kernel defined in a neighborhood of infinity. In case the compact sets are closures of quadrature domains the respective kernel is rational, enabling an effective matrix analysis algorithm for the non-overlapping decision. The simplest situation of two disks is presented in detail from a matrix model perspective as well as from a Riemann surface potential theoretic interpretation.

\end{abstract}


\section{Introduction}\label{sec:intro}

An analysis of Bergman space orthogonal polynomials on an ``archipelago", that is a collection of disjoint ``islands" contained in the complex plane, was carried out by the two of us in a 2009 joint paper with Ed Saff and Nikos Stylianopoulos \cite{Gustafsson-Putinar-Saff-Stylianopoulos-2009}. While natural questions of asymptotic analysis and constructive approximation theory were addressed there, in part as a Bergman space counterpart of a foundational contribution on Hardy space setting due to Widom \cite{Widom-1969}, we focus below on a different, related problem. Namely, to characterize in terms of power moments the non-overlapping of a given system of islands, that is compact subsets of the complex plane each endowed with the uniform area measure. Our quest is partially motivated by the prodigious activity recorded during the last decade on packing disks, ellipses, ovals, or even more general shapes \cite{Bennell, Kampas-2019, Kampas-2020}. We do not touch below the natural sequel question: once 
the islands are non-overlapping, is it possible to optimize their packing? On these topics see for instance \cite{Fasano}.

In the first part of the present article we adopt the ``hyponormal quantization" point of view. To be more specific, we treat the density function against the area measure of a compact set as the principal function of a hyponormal operator with rank one self-commutator (details given in the Preliminaries below). In other terms we treat the original non overlapping question as an inverse spectral problem. The Hilbert space counterpart of the dictionary provides a matrix positivity certificate for non-overlapping. One step further, we specialize the analysis to collections of 
quadrature domains (definition given in the Preliminaries) with the benefit of obtaining a rational, hence finitely determined, kernel which encodes in its 
positive definiteness the non overlapping criterion.

A second part of the article deals with the simplest, yet highly relevant, example of two disks. The iterative process of separating them via $2 \times 2$ matrix computations 
reveals the low computational complexity of our proposed algorithm. Moreover, we discuss in depth a gravi-equivalent deformation of two overlapping disks into 
an algebraic domain carrying a uniform area mass distribution. Recently developed methods of function theory on Riemann surfaces and potential theory naturally enter 
into discussion.

The contents is the following. Section 2 is devoted to preliminaries: the basics of spectral theory of hyponormal operators and a review of the concept of quadrature domain, presented from a geometric perspective. These two topics of the preliminaries merge into a block-matrix model encoding the {\it hyponormal quantization} of a quadrature domain. The main object of study in the present article is a four argument kernel described in detail in Section 3. A general framework for deriving non-overlapping certificates is presented in Section 4.  Section 5 specializes the results to the case of collections of quadrature domains. On that ground, an operator resolvent realization of the {\it rational} four argument kernel is included. Section 6 treats in detail, from at least four different points of view, the situation of two disks. First, the block-matrix model attached to a quadrature domain produces a matrix algebra algorithm with clear computational accessibility advantages. Then a dynamic, potential theoretic approach is developed, revealing this time the crucial role played by the Schwarz function in the moving boundaries analysis. A second dynamical interpretation, invoking this time an evolution by ``smashing" complements the previous computations. The article finishes by a change of metric transform (from Euclidean to spherical) underlying the importance of a larger group of motions for our non-overlapping criterion quest.
\bigskip

{\bf Acknowledgement.} The second author was partially supported by a Simons Collaboration Grant for mathematicians.

\section{Preliminaries}

\subsection{The exponential transform and its Hilbert space factorization}

Let $\zeta = x + iy$ denote the complex variable, with associated area measure
$$ dA(\zeta) = dx \wedge dy = \frac{ d \overline{\zeta} \wedge d\zeta}{2i}.$$
Let $g : \C \longrightarrow [0,1]$ be a (Lebesgue) measurable function of compact support. The associated {\it exponential transform} is:
\begin{equation}
E_g(w,z) = \exp ( \frac{-1}{\pi} \int_\C \frac{ g(\zeta) dA(\zeta)}{(\zeta-w)(\overline{\zeta} - \overline{z})}).
\end{equation}
Originally this function is defined for $z,w$ disjoint of the support of $g$. A separately continuous extension is possible, by adopting on the diagonal $z=w$ the convention $e^{-\infty} = 0.$
See for details  \cite{Gustafsson-Putinar-LNM}.

Taylor's expansion at infinity of $E_g$ contains, via an invertible generating series transform, the moment data:
$$ a_{kn}(g) = \int_\C g(\zeta) \zeta^k \overline{\zeta}^n dA(\zeta), \ \ k, n \geq 0.$$

One of the essential properties of the exponential transform
is the positive definiteness of the kernel
$$ (w,z) \mapsto 1-E_g(w,z), \ \ w,z \notin{\rm supp} (g).$$
In particular, for any collection of points $\lambda_j \notin {\rm supp} (g)$ and weights $c_j \in \C, 1 \leq j \leq n,$ one finds
$$ \sum_{j,k=1}^n E_g(\lambda_j, \lambda_k) c_j \overline{c_k} \leq 0,$$
whenever $\sum_{j=1}^n c_j = 0.$ In other terms
$E_g(w,z)$ is a {\it conditionally negative definite kernel}, a concept isolated by Schoenberg \cite{Schoenberg} from his studies of distance geometry. See also
\cite{Guella-Menegatto}.

The Hilbert space factorization of the kernel $1-E_g(w,z)$ goes back to the pioneering work of Joel Pincus \cite{Pincus-1968}, see also the Appendix in \cite{Gustafsson-Putinar-LNM} for details.
In a nutshell, there exists a constructive bijective correspondence between measures $g dA$, with  the weight $g : \C \longrightarrow [0,1]$ measurable function of compact support and irreducible
linear, bounded operators $T$ acting on a complex, separable Hilbert space $H$, constrained by the rank-one commutator condition:
$$ [T^\ast, T] = T^\ast T - T T^\ast = \xi \otimes \xi := \xi \langle \cdot, \xi \rangle, \ \ \xi \neq 0.$$ More specifically,
\begin{equation}\label{factor-E}
 E_g(w,z) = 1 - \langle (T^\ast - \overline{z})^{-1} \xi, (T^\ast - \overline{w})^{-1} \xi \rangle, \ \ w,z \notin {\rm supp} (g).
 \end{equation}
Note that the main objects of interest $E_g$ and $T$ depend only on the class of $g$ in $L^1(\C, dA)$. Operators satisfying the commutator identity $[T^\ast, T] \geq 0$ are called {\it hyponormal},
a class isolated by Halmos, reflecting a relaxation of the subnormality property.
In the dictionary above, the closed support of the function $g$ equals the spectrum of $T$. Moreover, if $g$ is the characteristic function of a bounded open set $U$, then
the operator $T-\lambda$ has Fredholm index equal to $-1$ for all points $\lambda \in U$. Following Pincus, the function $g$ associated to the hyponormal 
operator $T$ is called its {\it principal function}.

A rational approximation result, relevant for our study, is referring to principal functions which have only extremal values, $0$ and $1$. More precisely, 
let $T$ be an irreducible hyponormal operator with rank-one self-commutator $[T^\ast, T] = \xi \otimes \xi$ and principal function equal to the characteristic function of the spectrum
$\sigma(T)$.
Then the vector $\xi$ is $T$-rational cyclic, that is the subspace generated by $r(T)\xi$, with $r$ a rational function with poles outside $\sigma(T)$, is dense in $H$, cf.
Theorem 1 in \cite{Putinar-1988b}.
In case $g$ is the characteristic function of a domain with smooth, real analytic boundary, this density statement can be obtained via an analytic functional model of Hardy space type, proposed by Pincus, Xia and Xia. See again  \cite{Gustafsson-Putinar-LNM} for details.

The operator $T^\ast$ can be constructed from the exponential kernel as the multiplication by the variable
$\overline{z}$ on a Sobolev type functional space involving one partial derivative, \cite{Gustafsson-Putinar-LNM}.
In the present work we focus on an alternate, finite difference scheme, leading to an infinite matrix representation of the operator $T$ from the given data $E_g$.
To this aim, the following kernel, first appearing in  \cite{Putinar-1998} will be a main character in our story:
\begin{equation}\label{L-kernel}
 L_g(w,z; u,v) = \frac{ E_g(v,z)E_g(w,u) - E_g(w,z) E_g(v,u)}{(v-w) (\overline{u}-\overline{z}) E_g(w,u)}, \ \ w,z,u,v \notin  {\rm supp}(g).
 \end{equation}
 A separate section below is devoted to the properties of this four argument kernel.

\subsection{Quadrature domains}

A bounded domain $\Omega \subset \C$ is called a {\it quadrature domain} for analytic functions if there exists a distribution
$u \in {\mathcal D}'(\Omega)$ of finite support, satisfying
$$ \int_\Omega f dA = u(f),$$
for all complex analytic, integrable functions on $\Omega$. In other terms, a finite point quadrature formula is valid for the area measure carried by $\Omega$, {\it for all} analytic functions.
For example, Gauss mean value theorem applied to a disk $\D(a,r)$ provides such an universal formula:
$$ \int_{\D(a,r)} f dA = \pi r^2 f(a), \ \ f \in {\mathcal O}(\D(a,r)) \cap L^1(\D(a,r), dA).$$
Rational conformal transforms of the unit disk exhaust all simply connected quadrature domains.
Quadrature domains were introduced half a century ago by Aharonov and Shapiro \cite{Aharonov-Shapiro-1976}. The concept turned out to be central for studies of conformal mapping (the original motivation), fluid mechanics, potential theory, integrable systems and operator theory. See the survey \cite{Gustafsson-Shapiro-2005} and the entire volume containing it for a glimpse on the fascinating evolution and ramifications of quadrature domain theory.

The boundary of a quadrature domain $\Omega$ is a real algebraic curve. More precisely, there exists an irreducible hermitian polynomial $R(z, \overline{z}) = \overline{ R(z, \overline{z})}$ with the property:
$$ \Omega = \{ z \in \C; R(z, \overline{z}) < 0 \},$$
modulo some finite set. Denoting by $P(z)$ the minimal monic polynomial of degree $d$ vanishing at the quadrature nodes, we find
$$R(z,\overline{z}) = |P(z)|^2 - |P_{d-1}(z)|^2 - |P_{d-2}(z)|^2 - \ldots - |P_1(z)|^2 - |P_0(z)|^2,$$
where $P_j \in \C[z]$ are polynomials constrained by the degree value $\deg P_j = j, \ 0 \leq j < d.$ Details can be found in
\cite{Gustafsson-Putinar-LNM}. The possible singular points of the boundary of $\Omega$, contained in the zero locus of the 
polynomial $R(z,\overline{z})$ are special, and well charted. Also, possible relevant for a continuation of the present article, is the fact that quadrature domains are dense with respect to Hausdorff metric among all planar domains.

For terminology reasons, paying attention to kernel functions, we denote henceforth the polarization of the hermitian polynomial above as:
$$ Q(z,w) = R(z, \overline{w}), \ \ z,w \in \C.$$
To the effect that $Q$ is a polynomial in real variables, analytic with respect to $z$ and anti-analytic with respect to $w$.

A characteristic feature of quadrature domains is the possibility of extending to a meromorphic function in $\Omega$
the Schwarz function of the boundary,
that is an analytic function $S(z)$ defined in a neighborhood of $\partial \Omega,$ satisfying $S(\zeta) = \overline{\zeta},
 \zeta \in \partial \Omega$. If $R(z,\overline{z}) = 0$ defines the boundary of the quadrature domain, then $R(z, S(z)) = 0$ for $z$ belonging to a neighborhood of $\partial \Omega$. In particular $S(z)$ is in this case an algebraic function.
 For the disk $\D(a,r)$ one finds $S(z) = \overline{a} + \frac{r^2}{z-a}$. Quite remarkable, the role of Schwarz function in mathematical analysis and geometry continues to expand \cite{Davis-1974, Shapiro-1992}.

By abuse of terminology, we drop the connectedness assumption in the definition of a quadrature domain and allow disjoint unions of open sets subject to the same finite point universal quadrature identity for analytic functions. 

The exponential transform and the hyponormal quantization of a ``shade function"
detect quadrature domains, as follows. The measurable function $g: \C \longrightarrow [0,1]$ of compact support coincides with the characteristic function of a quadrature domain $\Omega$ if and only if its exponential transform $E_g(w,z)$ is rational in a neighborhood of infinity, with a polarized denominator of the form $P(w) \overline{P(z)},$ where $P(z)$ is a polynomial of degree $d > 0$. In this case 
$$ P(w) \overline{P(z)} E_g(w,z) = Q(w,z), \ \ |w|, |z| >>1,$$
the numerator $Q(w,z)$ gives the defining equation of the set $\Omega$. The zeros of $P$ identify, multiplicity counted, the quadrature nodes. On the other hand, the irreducible hyponormal operator $T$ with principal function $g$ and self-commutator
$[T^\ast, T] = \xi \otimes \xi$  satisfies the finiteness condition\\
 ${\rm dim}\  {\rm span} \{ T^{\ast k}\xi, \ k \geq 0\} < \infty$ if and only if 
$g$ is the characteristic function of a quadrature domain. Details can be found in \cite{Putinar-1990, Gustafsson-Putinar-LNM}.

When regarded from moment data, quadrature domains are finitely determined, in a precise sense implemented by the exponential transform and the matrix model offered by hyponormal quantization. The following formula encodes in a single line the above observations:
\begin{equation}
\frac{Q(z,w)}{P(z) \overline{P(w)}} =  \exp ( \frac{-1}{\pi} \int_{ Q(\zeta, \zeta)< 0}  \frac{ dA(\zeta)}{(\zeta-z)(\overline{\zeta} - \overline{w})}),\ \ |z|, |w| >>1,
\end{equation}
where $\{ z \in \C; Q(z,z)<0\}$ is a quadrature domain with nodes at the zeros of the minimal, monic polynomial $P(z)$.

\subsection{The two diagonal block-matrix model}

Let $\Omega \subset \C$ be a quadrature domain. We denote by $T$ its hyponormal quantization, that is an irreducible Hilbert space operator $T \in {\mathcal L}(H)$ with rank-one self commutator
$[T^\ast, T] = \xi \langle \cdot, \xi\rangle$ and principal function equal to the characteristic function of $\Omega$. 
In this case the $T^\ast$- cyclic subspace $H_0 = {\rm span} \{ T^{\ast k} \xi, \ k\geq 0\}$ is finite dimensional. The minimal polynomial of the restriction
$D_0 = T^\ast|_{H_0}$ coincides with the monic polynomial $P(z)$ of degree $d$ vanishing at the quadrature nodes \cite{Gustafsson-Putinar-LNM}. Due to the invariance of the principal function $g_T = \chi_{\Omega}$ under finite rank-perturbations, it follows that the subspaces
$$ H_k = {\rm span} \{ T^j x, \ 0 \leq j \leq k, \ x \in H_0 \}$$ have exact dimension
$$ \dim H_k = (k+1) d, \ \ k \geq 0.$$
The filtration $H_k \subset H_{k+1}$ reveals a block-matrix weighted shift structure of the operator $T$:
first the entire space $H$ can be decomposed into an orthogonal direct sum:
$$ H = H_0 \oplus [H_1 \ominus H_0] \oplus [H_2 \ominus H_1] \oplus \ldots,$$
and correspondingly one can write:
\begin{equation}\label{weighted-shift}
T = \left( \begin{array}{ccccc}
             D_0 & 0 & 0 & \ldots & \ldots\\
             A_0 & D_1 & 0 & & \ldots\\
             0& A_1 & D_2 & 0 & \ldots\\
             0 & 0 & A_2 & D_3 & \ldots\\
             \vdots & & \ddots & \ddots & \\
             \end{array} \right).
\end{equation}
A convenient choice of orthogonal bases in each summand $H_{k+1} \ominus H_k = \C^d$ allows us to assume $A_k > 0$ for all $k \geq 0.$
The commutation relation $[T^\ast, T] = \xi \langle \cdot, \xi \rangle$ is equivalent to the recurrent system of equations
\begin{equation}\label{recurrence}
[D_k^\ast , D_k] + A_{k}^2 = A_{k-1}^2, \ \ k \geq 0; \ \ A_{-1} = \xi \langle \cdot, \xi \rangle,$$ $$
A_k D_{k+1} = D_{k} A_k, \ \ k \geq 0.
\end{equation} 
Note that ${\rm trace}  \ A^2_k = {\rm trace} \ A^2_{k-1}, \ \ k \geq 0,$ hence
$$ {\rm trace} \ A^2_k  = \|\xi \|^2 = \frac{ {\rm Area} (\Omega)}{\pi},$$
that is the off diagonal entries in (\ref{weighted-shift}) are uniformly bounded in norm. Consequently, the boundedness of the entire operator $T$ is equivalent to 
$ \sup_k \| D_k \| < \infty.$ Moreover, all diagonal entries $D_k$ are similar to each other, hence iso-spectral, and bounded in norm by that of $T$:
$$ \|D_k \| \leq \|T \|, \ \ k \geq 0.$$
Details can be found in \cite{Gustafsson-Putinar-LNM} and the references cited there.
We adopt the following notation, related to the orthogonal decomposition of the Hilbert space $H$:
$$ D = {\rm diag} (D_0, D_1, D_2, \ldots), \ \ A = {\rm diag} (A_0, A_1, A_2, \ldots),$$
while $S$ stands for the shift operator:
$$ S (x_0, x_1, x_2, \ldots) = (0, x_0, x_1, \ldots), \ \ x_k \in H_k \ominus H_{k-1}, \ \ k \geq 0; \ \ H_{-1} = 0.$$
The staire-case matrix decomposition $(\ref{weighted-shift})$ becomes 
$$ T = D + S ,$$
with resolvent:
$$ (T-w)^{-1} = (D-w + SA)^{-1} = (I + (D-w)^{-1} SA)^{-1} (D-w)^{-1}, \ \ |w| > \| T \|.$$
When applied to an element $x \in H_0$ of the initial space, each term of the convergent Neumann series:
$$  (T-w)^{-1}x = \sum_{k=0}^\infty (-1)^n [(D-w)^{-1} SA]^n (D-w)^{-1} x = $$
$$ \sum_{k=0}^\infty (-1)^n (D_n -w)^{-1} A_{n-1} (D_{n-1} -w)^{-1} A_{n-2} \ldots (D_0-w)^{-1}x,$$
belongs to a different orthogonal component of $H$. The result is a sums of squares decomposition which gives additional structure to the positivity of the kernel $L$.
Indeed, for $|z|, |w| > \| T \|$, and after identifying the summand $H_n \ominus H_{n-1}$ with $\C^d$ one defines
$$ f_n(w,z) = (D_n -w)^{-1} A_{n-1} (D_{n-1} -w)^{-1} A_{n-2} \ldots (D_0-w)^{-1} (D_0^\ast - \overline{z})^{-1} \xi.$$
Then
\begin{equation}\label{direct-summand}
 f_{n+1}(w,z) = (D_{n+1}-w)^{-1} A_n f_n(w,z), \ n\geq 0,$$
 $$ f_0(w,z) = (D_0-w)^{-1} (D_0^\ast - \overline{z})^{-1} \xi.
 \end{equation}
 
Note that the parent functions $F_n =  [(D-w)^{-1} SA]^n (D-w)^{-1} (D^\ast - \overline{z})^{-1} \xi \in H_n \ominus H_{n-1}, \ n \geq 1,$ are mutually orthogonal and the action of the block matrix yields a three term relation:
$$ T F_n(w,z) = A_{n-1} F_{n-1}(w,z) + w F_n(w,z) + A_n F_n(w,z), \ n \geq 1.$$
 In conclusion
 \begin{equation}\label{double-resolvent}
 (T-w)^{-1}(T^\ast - \overline{z})^{-1} \xi = \sum_{k=0}^\infty (-1)^k F_k(w,z), \ \ |w|, |z| > \|T \|,
  \end{equation}
 with an absolutely convergent series.

\section{The four argument kernel}

Let $T$ be an irreducible hyponormal operator with rank-one self-commutator $[T^\ast, T] = \xi \otimes \xi$, with principal function $g$. The four argument kernel (\ref{L-kernel}) can be factored through
the underlying Hilbert as follows:
\begin{equation}\label{factor-L}
 L(w,z; u,v) = \langle (T-v)^{-1} (T^\ast - \overline{u})^{-1} (T-w)^{-1} (T^\ast - \overline{z})^{-1} \xi, \xi \rangle, \ \ w,z,u,v \notin {\rm supp} (g).
\end{equation}
Indeed, it suffices to repeatedly apply to (\ref{factor-E}) the resolvent equation
$$ (S-\lambda)^{-1} - (S-\mu)^{-1} = (\lambda-\mu) (S-\lambda)^{-1} (S-\mu)^{-1}.$$

Thus, $L$ is a positive definite kernel defined on the resolvent polydomain $V \times V$, where $ V = (\C \setminus \sigma(T)) \times (\C \setminus \sigma(T))$.
The main observation of this section is that the germ at infinity of $L$, or at any other point of $V$ not only determines constructively $T$, but it reveals the characteristic property of the conditionally positive definite kernel $E$, to be associated with a measurable function $g$ with values in the interval $[0,1]$.
The inner structure of the kernel $L$ reflects the commutator identity $[T^\ast,T] = \xi \otimes \xi$. 

In general, the {\it positive definiteness} of a kernel $M(w,z; u,v)$ defined on a general polydomain  $U^4$ means that for every finite collection
of points $w_k, z_k \in U$ and weights $\lambda_k \in \C, \ 1\leq k \leq m,$ one has:
$$\sum_{k,\ell = 1}^m M(w_k, z_k; w_\ell, z_\ell) \lambda_k \overline{\lambda_\ell} \geq 0.$$
Then we write in short $M \succeq 0$. To be accurate, the condition above is known as {\it positive demi-definiteness}, allowing equality to zero for some non-trivial selections of points and weights.
Unless essential in the proofs, throughout this article we simply call positive or positive definite a positive semi-definite kernel.
For $R >0$ we set $\Delta_R = \{ \zeta; |\zeta| > R\}.$

\begin{theorem}\label{free-infinity} Let $g: \C \longrightarrow \R$ be a measurable function with compact support and let $E = E_g$ be the associated exponential transform.
The following conditions are equivalent:
\begin{enumerate}
 \item{ The function $g$ takes values in the interval $[0,1]$, almost everywhere with respect to area measure;}
 \item{ There exists a constant $C>0$ such that the following positivity conditions are fulfilled:
 \begin{equation}\label{bounded}
L(w,z; u,v) \succeq 
 C [w \overline{u} L(w,z; u,v) - w \overline{u} L(w,z; u,v)|_{w=\infty} - $$ $$ -w \overline{u} L(w,z; u,v)|_{u=\infty}  + w \overline{u} L(w,z; u,v)|_{u,w = \infty}] \succeq 0.
\end{equation}
}
 \item{ The kernel $L$ is positive definite in a polydisk $\Delta_R^4$ centered at infinity and there exists a point $\lambda \in \Delta_R$
 together with a constant $C>0$, such that
 \begin{equation}\label{point-evaluation}
 C \frac{ L(\lambda, z; \lambda, v) - L(w,z; \lambda, v) - L(\lambda, z; u, v) + L(w,z; u,v)}{(\lambda-w) (\overline{\lambda}-\overline{u})} \succeq L(w,z; u,v);
 \end{equation}}
 \item{There exists an open set $U$ disjoint of the support of $g$ such that the kernel $L$ is positive definite on $U^4$ and condition (\ref{point-evaluation})
 holds for a point $\lambda \in U$ and a constant $C>0$.}
 \end{enumerate}
 \end{theorem} 

\begin{proof} The equivalence between 1) and 2) is established in Theorem 4.1 in \cite{Putinar-1998}. In turn these statements are equivalent to the existence of an irreducible 
hyponormal operator
$T$ with rank-one self-commutator and principal function equal to $g$. The two positivity conditions in 2) quantify the boundedness of $T$ acting inside the
Gram matrix 
$\langle T^k T^{\ast \ell} \xi, T^m T^{\ast n} \xi \rangle, \ \ k, \ell, m, n \geq 0,$ plus the commutator identity $[T^\ast, T] = \xi \otimes \xi$. Conditions 3) and 4) restate the same two properties 
on the Taylor expansion at different points than infinity, of the correlation kernel of resolvents $\langle (T-w)^{-1} (T^\ast - \overline{z})^{-1} \xi, (T-u)^{-1} (T^\ast - \overline{v})^{-1} \xi \rangle$.

Assume assertion 3) holds and consider the 
Kolmogorov factorization of the positive definite kernel $L$. Specifically, there exists a complex Hilbert space and a map $\rho: \Delta_R \longrightarrow H$
satisfying
$$ \langle \rho(w,z), \rho(u,v) \rangle = L(w,z; u,v), \ w,z,u,v \in \Delta_R.$$
Consequently, $\rho$ is an analytic function in $w$ and anti-analytic in $z$, vanishing with respect to each variable at infinity. Without loss of generality we can assume that the Hilbert space $H$ is generated by the vectors $\rho(w,z), \ w,z \in \Delta_R$.
Condition (\ref{bounded}) is equivalent to a majorization of Gram kernels:
$$  \langle \rho(w,z), \rho(u,v) \rangle \succeq C \langle w \rho(w,z) - w \rho(w,z)|_{w=\infty}, u \rho(u,v) - u \rho(u,v)|_{u=\infty} \rangle.$$
That is, there exists a linear bounded operator $T : H \longrightarrow H$ with the property
$$ T  \rho(w,z) = w \rho(w,z) - w \rho(w,z)|_{w=\infty}, \ \ w,z \in \Delta_R.$$
It turns out then that the inner structure of the kernel $L$, inherited from the exponential transform $E_g,$ implies $[T^\ast, T] = \xi \otimes \xi$ and
$$ \rho(w,z) = (T-w)^{-1} (T^\ast - {\overline z})^{-1} \xi.$$
See \cite{Putinar-1998} for full details.

The hypothesis (\ref{point-evaluation}) is equivalent to:
$$ C \langle \frac{ \rho(\lambda,z) - \rho(w,z)}{\lambda - w}, \frac{ \rho(\lambda,v) - \rho(u,v)}{\lambda - u} \rangle \succeq L(w,z; u,v),$$
on the set $\Delta_R^4$. As before, there exists a linear bounded operator $X$ mapping one frame of vectors to the other:
$$ X  \frac{ \rho(\lambda,z) - \rho(w,z)}{\lambda - w} = \rho(w,z).$$
Therefore
\begin{equation}\label{resolvent}
 X  (\rho(\lambda,z) - \rho(w,z)) = (\lambda-w) \rho(w,z).
 \end{equation}
Evaluating at $w = \infty$ we find
$$ X \rho(\lambda,z) = - w \rho(w,z)|_{w=\infty}.$$
Consequently,
$$ (X+ \lambda) = w \rho(w,z)  - w \rho(w,z)|_{w=\infty}.$$
This proves $T = X + \lambda$ and statement 2) is verified.

Vice-versa, assume assertions 1) and 2) are true.  Then the linear bounded operator $T$ satisfying $[T^\ast, T] = \xi \otimes \xi$ produces the linear transform
$X = T-\lambda$. Then $X$ fulfills the above identities with $\rho(w,z) = (T-w)^{-1}(T^\ast - \overline{w})^{-1}\xi.$ The boundedness of $X$ and the commutator identity
$[X^\ast, X] = \xi \otimes \xi$ are encoded in condition 3).

Similarly, under assumption 4) we start with the above factorization:
$$ \langle \rho(w,z), \rho(u,v) \rangle = L(w,z; u,v), \ w,z,u,v \in U,$$
and repeat, defining the linear bounded map $X: H \longrightarrow H$ via
the same identity (\ref{resolvent}). This implies
$$ (X + \lambda) [\rho(w_1,z) - \rho(w_2,z)] =  [w_1\rho(w_1,z) - w_2\rho(w_2,z)], \ \ w_1,w_2, z \in U.$$
That is
$$ (X+\lambda-w_1) \rho(w_1,z) = (X+\lambda-w_2) \rho(w_2,z), \ \ w_1,w_2, z \in U.$$
Well known elementary algebra streaming from the resolvent equation implies that $\rho(w,z) = (X+ \lambda-w)^{-1} \sigma(z),$
where $\sigma: U \longrightarrow H$ is an antianalytic map. The structure of the kernel $L$ enters into the picture showing
that $T + X+ \lambda$ is the pure hyponormal operator with rank-one self-commutator factoring the function $\rho$ as before. See again \cite{Putinar-1998} for 
more details.

\end{proof} 

We recognize in the second term of the positivity condition above a double finite difference at the point at infinity.
By its construction, and the basic properties of the exponential transform we note that  $L(w,z; u,v)$ vanishes of the first order at infinity, in every variable separately.
Specifically,
$$ L(w,z; u,v) = \frac{\gamma} {w\overline{z} v \overline{u}} + {\rm higher \ order \ terms},$$
with leading term coefficient $\gamma = \frac{{\rm Area}}{\pi}$.

To simplify notation, the standard finite difference notation: $f [a,b] = \frac{f(a) - f(b)}{a-b}$ is adopted from now on. In case the function depends on two variables we write
$ f([a,b], c) = \frac{ f(a,c) - f(b,c)}{a-b}$, see for instance \cite{Gelfond}. Then condition (\ref{point-evaluation}) simplifies to:
$$ C L([\lambda,w],z; [\lambda,u], v) \succeq L(w,z; u,v).$$

Further on, identifying finite difference expressions:
$$ L(w,z; u,v) = \frac{ E(v,z) \frac{E(w,u)-E(w,z)}{\overline{u}-\overline{z}} - E(w,z) \frac{E(v,u) - E(v,z)}{\overline{u}-\overline{z}} }{(v-w)E(w,u)},$$
we derive a determinantal formula for the kernel:
$$ (v-w)(\overline{u}-\overline{z}) E(w,u)  L(w,z; u,v) = \left| \begin{array}{cc}
                E(v,z) & E(w,z)\\
                 E(v,u) & E(w,u)\\
                 \end{array} \right|.$$
Operations on rows and columns yield the following result.

\begin{proposition}\label{FD}
The $L$ kernel associated to the exponential transform $E$ can be written in terms of divided differences as:
\begin{equation}\label{finite-diff}
 L(w,z; u,v) = \frac{-1}{E(w,u)} \left| \begin{array}{cc}
                E([v,w],[z,u]) & E(w,[z,u])\\
                 E([v,w],u) & E(w,u)\\
                 \end{array} \right|, \ z,w,u,v \in \Delta_R.
\end{equation}
\end{proposition} 

The value of $L$ on the diagonal is obviously non-negative:
$$ L(w,z; w,z) \geq 0, \ \ w,z \in \Delta_R.$$
In particular, the numerator appearing in the definition of this kernel is non-negative:
$$ |E(w,z)|^2 - E(w,w) E(z,z) \geq 0, \ \ w,z \in \Delta_R.$$
In other words, the exponential function satisfies a {\it reverse Cauchy-Schwarz inequality} at infinity, and as a matter of fact
even on the complement of the support of the defining function $g$. 
This well known property is encoded in the positive semi-definiteness of the kernel $\frac{1}{E(w,z)},$ see for details \cite{Gustafsson-Putinar-LNM}.
Equivalently, the finite difference expression of the kernel $L$ implies on the diagonal $(w,z) = (u,v)$ the inequality:
$$ \left| \begin{array}{cc}
                E([z,w],[z,w]) & E(w,[z,w])\\
                 E([z,w],u) & E(w,w)\\
                 \end{array} \right| \leq 0, \ z,w,u,v \in \Delta_R.$$

Similarly, by passing to the limit $v \rightarrow w$ and $u \rightarrow z$ in relation (\ref{finite-diff}) we obtain the values of $L$ on the anti-diagonal:
$$ L(w,z; z,w) = - E(w,z) 
\frac{\partial}{\partial w} 
\frac{ 
\frac{\partial E}{\partial \overline{z}}
}
{E} (w,z).$$

\begin{corollary}\label{flip} In the conditions of Theorem (\ref{free-infinity}), the kernel
$$ (z,w) \mapsto  L(w,z; z,w), \ \ z,w \in \Delta_R,$$
is positive semi-definite. 
 \end{corollary}
 
 \begin{proof} 
 Indeed the basic commutator identity and the resolvent equation yield:
 $$ L(w,z; z,w) =  \langle (T-w)^{-1} (T^\ast - \overline{z})^{-1} (T-w)^{-1} (T^\ast - \overline{z})^{-1} \xi, \xi \rangle = $$
 $$ \langle (T^\ast - \overline{z})^{-2} \xi , (T^\ast - \overline{w})^{-2} \xi \rangle + \langle (T-w)^{-1}\xi, (T-z)^{-1}\xi \rangle \langle (T^\ast - \overline{z})^{-1} , (T^\ast - \overline{w})^{-1} \xi \rangle.$$
 The second term is a product of two positive definite kernels. In view of Schur product theorem, it is also positive semi-definite.
  \end{proof} 
 
We also record for later use the following consequence of Proposition \ref{FD}.
 
 \begin{corollary} For every $z,w,u,v \in \Delta_R$ one finds:
 \begin{equation}
  L(w,z; u,v) = $$ $$ \frac{ \langle (T^\ast - \overline{u})^{-1} (T^\ast - \overline{z})^{-1} \xi, (T^\ast - \overline{w})^{-1} \xi \rangle \langle (T^\ast - \overline{u})^{-1} \xi, (T^\ast - \overline{v})^{-1} (T^\ast - \overline{w})^{-1} \xi \rangle}{1- \langle (T^\ast - \overline{u})^{-1} \xi, (T^\ast - \overline{w})^{-1} \xi \rangle } + $$ $$
  + \langle (T^\ast - \overline{u})^{-1} (T^\ast - \overline{z})^{-1} \xi, (T^\ast - \overline{v})^{-1} (T^\ast - \overline{w})^{-1} \xi \rangle.
  \end{equation}
  \end{corollary}
 
\begin{proof} The finite differences combined with the resolvent equation imply identities of the form
$$ E(w,[z,u]) = - \frac{  \langle (T^\ast - \overline{z})^{-1} \xi, (T^\ast - \overline{w})^{-1} \xi \rangle - \langle (T^\ast - \overline{u})^{-1}  \xi, (T^\ast - \overline{w})^{-1} \xi \rangle}
{\overline{z} - \overline{u}} =  $$ $$ - \langle (T^\ast - \overline{u})^{-1} (T^\ast - \overline{z})^{-1} \xi, (T^\ast - \overline{w})^{-1} \xi \rangle.
$$

\end{proof} 
\bigskip

{\bf Example.} We end this section by some simple computations revealing the structure of the four argument kernel on a disk.
 For the unit disk we have:
$$ E_{\D}(w,z) = 1 - \frac{1}{w \overline{z}}, \ \ |z|, |w| > 1.$$
Accordingly,
$$ L_\D (w,z; u,v) = \frac{ (1 - \frac{1}{v \overline{z}}) (1 - \frac{1}{w \overline{u}}) - (1 - \frac{1}{w \overline{z}})(1 - \frac{1}{v \overline{u}})}{(v-w) (\overline{u}-\overline{z}) (1 - \frac{1}{w \overline{u}})} =$$
$$ \frac{ (v\overline{z}-1) (w\overline{u}-1) - (w\overline{z}-1) (v\overline{u}-1) } {(v-w) (\overline{u}-\overline{z}) (w\overline{u}-1) v \overline{z}} =
\frac{1}{w\overline{z} v \overline{u}} \cdot \frac{1}{1 - \frac{1}{w \overline{u}}}.$$
The finite differences appearing in formula (\ref{finite-diff}) read as follows:
$$ E_{\D}(w,[z,u]) = \frac{1}{w \overline{z} \overline{u}}, \ \  E_{\D}([v,w],u) = \frac{1}{v w \overline{u}},$$
respectively
$$ E_{\D}( [v,w], [z,u]) = \frac{-1}{v w \overline{z} \overline{u}}.$$
The numerator in formula (\ref{finite-diff}) becomes 
$$  \frac{-1}{v w \overline{z} \overline{u}} +  \frac{1}{v w^2 \overline{z} \overline{u}^2} - \frac{1}{v w^2 \overline{z} \overline{u}^2} =  \frac{-1}{v w \overline{z} \overline{u}} ,$$
which validates our prior computations.

A disk of center $a$ and radius $r$ will produce the following kernel:
\begin{equation}\label{disk}
L_{\D(a,r)} (w,z; u,v) = \frac{r^2}{(w-a) (\overline{z} - \overline{a}) (v-a)( \overline{u}- \overline{a})} \cdot 
\frac{1}{1 - \frac{r^2}{(w-a)( \overline{u}- \overline{a})}}.
\end{equation}
To the extent that the double finite difference
$$ L_{\D(a,r)} (w,z; u,v) E_{\D(a,r)}(w,u) = \frac{r^2}{(w-a) (\overline{z} - \overline{a}) (v-a)( \overline{u}- \overline{a})} $$
is remarkably simple.

\section{The islands do not overlap}

\subsection{The general framework}
Let $K_j, \ j \in J,$ be a finite collection of ``thick" compact subsets of the complex plane. By thick meaning that
every $K_j$ is the closed support of the characteristic function $\chi_{K_j}$ seen as an element of $L^1(\C, dA)$.
We are interested in an effective certificate assuring that the ``islands" $K_j, j \in J,$ do not overlap. 
More specifically, that means that all intersections $K_j \cap K_\ell, \ j \neq \ell,$ have zero area measure.

Denote by $E_j(w,z)$ the exponential transform associated to $K_j$, analytic in $w$, anti-analytic in $\overline{z}$, defined for $z,w$ in a neighborhood of infinity.

\begin{theorem}\label{archipelago}  Let $K_j, j \in J$ denote finitely many thick compact subsets of $\C$. Denote $E(w,z) = \prod_{j \in J} E_j(w,z)$.
The islands $K_j, j \in J,$ do not overlap if and only if there exists a positive constant $C$, with the property that the kernel
\begin{equation}
L(w,z; u,v) = \frac{ E(v,z)E(w,u) - E(w,z) E(v,u)}{(v-w) (\overline{u}-\overline{z}) E(w,u)}
\end{equation}
satisfies one of the equivalent positive semi-definiteness constraints stated in Theorem \ref{free-infinity}.
\end{theorem} 

In view of Corollary \ref{flip},  the logarithmic derivative provides a necessary positivity condition for non-overlapping. Indeed, for
$w,z$ outside the union of the islands, the kernel
$$  L(w,z; z,w) = - \prod_{j \in J} E_j(w,z) \sum_{j \in J}  \frac{\partial}{\partial w} \frac{ \frac{\partial E_j}{\partial \overline{z}}}{E_j} (w,z)$$
is positive definite whenever ${\rm Area} (\cup_{j \neq \ell} K_j \cap K_\ell) = 0.$

Several other necessary conditions can be derived from restricting the values of parameters in Theorem \ref{archipelago}. One of them is 
the positive definiteness of the kernel
$$1-E(v,z) = w \overline{u} L(w,z; u,v)|_{u,w = \infty}.$$ We prove later on that this criterion does not separate even two disks. Even more, the example below shows that the denominator in the expression of the kernel $L$ cannot be dropped for non overlapping purposes.\bigskip

{\bf Example.} The hyponormal operator with rank-one self-commutator associated to the {\it ellipse} described parametrically by $\{ r \overline{z} + z; \ |z| =1\},$ where $r \in (0, 1)$ is
$r S^\ast + S.$ Here $S : \ell^2 \longrightarrow \ell^2$ stands for the unilateral shift, see Section 7.3.1 in \cite{Gustafsson-Putinar-LNM}.

We show that, in general the kernel $L(w,z; u,v)E(w,u)$ is not positive semi-definite. Indeed, in the hyponormal quantization setting
$$ L(w,z; u,v)E(w,u) =  \langle (T-v)^{-1} (T^\ast - \overline{u})^{-1} (T-w)^{-1} (T^\ast - \overline{z})^{-1} \xi, \xi \rangle \cdot $$ $$\cdot (1 -  \langle (T-w)^{-1} (T^\ast - \overline{u})^{-1} \xi, \xi \rangle).$$
The coefficient $c(v,z)$ of $w^{-2} \overline{u}^{-2}$ in the power series expansion at infinity is
$$ c(v,z)  = \langle T (T^\ast - \overline{z})^{-1} \xi, T (T^\ast - \overline{v})^{-1} \xi \rangle -  \langle (T^\ast - \overline{z})^{-1} \xi, (T^\ast - \overline{v})^{-1} \xi \rangle \| \xi \|^2.$$
If $L(w,z; u,v)E(w,u)$ were positive semi-definite, then $c(v,z)$ would be, too. But that implies that the operator $T$ is injective with closed range when restricted to the linear span of the vectors
$T^{\ast k}\xi, \ \ k \geq 0.$ 

We consider $T = r S^\ast + S - r - 1$. Let $e_0, e_1, e_2, \ldots$ be the standard orthonormal basis of $\ell^2$. The self-commutator is $[T^\ast, T] = (1-r^2) e_0 \langle \cdot, e_0\rangle$.
The vectors $e_0, T^\ast e_0, T^{\ast 2} e_0, \ldots$ span $\ell^2$. Moreover, the point $r+1$ belongs to the ellipse, which is also the boundary of the spectrum of $rS^\ast + S$. Since the operator $T$
has the point $z=0$ in the boundary of its spectrum, the range of $T$ is not closed. The preceding general observations imply that the kernel $L(w,z; u,v)E(w,u)$ associated to an ellipse is not positive semi-definite.
\bigskip

\subsection{Real algebra certificate of non-overlapping}

Assume that in the conditions of Theorem \ref{archipelago}, the islands $K_j$ are semialgebraic sets each defined by a single polynomial function. Then powerful methods of Real Algebraic Geometry enter into the game. 
We reproduce below, adapted  to our setting, a Positivstellensatz due to Stengle
\cite{Stengle-1974}. 

For the result below only, we assume $x = (x_1, x_2, \ldots, x_d)$ is the variable in $\R^d$ with $d \geq 1.$ We denote by $\Sigma^2 \subset \R[x]$ the convex cone of sums of square of polynomials with real coefficients.

\begin{proposition} Let $p_j \in \R[x], j \in J,$ be a finite set of non-constant polynomials depending on $d$ real variables. The sub-level sets
$\Omega_j = \{ x \in \R^d, p_j(x) > 0 \}, j \in J,$  are mutually disjoint if and only if for every distinct pair of indices $j, k \in J,$ there exist positive integers $p,q$ with the propery:
$$  -f_j^{2p} f_k^{2q} \in \Sigma^2 + f_j \Sigma^2 + f_k \Sigma^2 + f_j f_k \Sigma^2. $$

\end{proposition}

Stengle's certificate is an offspring of the celebrated Tarski's elimination of quantifiers theorem.  The monograph \cite{BCR} gives an authoritative perspective on both real algebra and real algebraic geometry, until 1990-ies.
When applied to the $L$ kernel associated to principal semi-algebraic sets contained in $\R^2$, the main difference between Stengle's condition of non-overlapping and Theorem \ref{archipelago} is the array of algebra identities involving polynomials of unspecified degree, versus a single kernel positivity criterion.

\subsection{Merging islands}
For the rest of this section we concentrate on analyzing the process of adding a separate island to an archipelago.

\begin{proposition} Let $K_1$ and $K_2$ be two (area measure) disjoint, thick compact sets contained in the complex plane.
The corresponding four term kernels are denoted $L_j$, while the exponential kernels are $E_j, j =1,2.$
Then for sufficiently large values of $u,v,w,z,$ the kernel corresponding to the union $K_1 \cup K_2$ is:
\begin{equation}\label{merging-1}
L(w,z; u,v) = L_1(w,z; u,v) E_2(v,z) + L_2(w,z; u,v) E_1(v,z) - $$ $$(v-w)(\overline{u}-\overline{z}) L_1(w,z; u,v) L_2(w,z; u,v).
\end{equation}

\end{proposition} 

\begin{proof} By its definition, the kernels $E_j(v,z)$ are invertible for large values of the two arguments. One can write in general
$$\frac{(v-w) (\overline{u}-\overline{z})L(w,z; u,v)}{E(v,z)} = 1- \frac{E(w,z) E(v,u)}{E(v,z) E(w,u)}.$$
Consequently, 
$$ 1 - \frac{(v-w) (\overline{u}-\overline{z})L(w,z; u,v)}{E(v,z)}$$ is multiplicative as a set function (of the generating thick compact $K$).
Whence
$$ \frac{(v-w) (\overline{u}-\overline{z})L(w,z; u,v)}{E(v,z)} = \frac{(v-w) (\overline{u}-\overline{z})L_1(w,z; u,v)}{E_1(v,z)} + $$ $$ 
\frac{(v-w) (\overline{u}-\overline{z})L_2(w,z; u,v)}{E_2(v,z)} - \frac{(v-w)^2 (\overline{u}-\overline{z})^2 L_1(w,z; u,v) L_2(w,z; u,v)}{E_1(v,z) E_2(v,z)}.$$
Since all kernels are analytic/antianalytic in the respective variables, one can divide the common factor $(v-w) (\overline{u}-\overline{z})$. Moreover,
$E = E_1 E_2$ and the proof is complete.
\end{proof} 

Some equivalent forms of identity (\ref{merging-1}) are:
\begin{equation}\label{merging-2}
 \frac{L(w,z; u,v)}{E(v,z)} = \frac{L_1(w,z; u,v)}{E_1(v,z)}  + \frac{L_2(w,z; u,v)}{E_2(v,z)}  -  $$ $$ 
(v-w) (\overline{u}-\overline{z})    \frac{L_1(w,z; u,v)}{E_1(v,z)}  \frac{L_2(w,z; u,v)}{E_2(v,z)} ,
\end{equation}
or
\begin{equation}\label{merging-3}
L(w,z; u,v) = L_1(w,z; u,v) E_2(v,z) + \frac{E_1(w,z) E_1(v,u)}{ E_1(w,u)} L_2(w,z; u,v) .
\end{equation}
Dividing by $E(v,z)$ one finds:
\begin{equation}\label{merging-4}
\frac{L(w,z; u,v)}{E(v,z)} = \frac{L_1(w,z; u,v)}{E_1(v,z)}  + \frac{E_1(w,z) E_1(v,u)}{ E_1(w,u)E_1(v,z)} \frac{L_2(w,z; u,v)}{E_2(v,z)}.
\end{equation}
Finally, multiplying equation (\ref{merging-3}) by $E(w,u)$ one obtains:
\begin{equation}\label{merging-5}
L(w,z; u,v)E(w,u) = [L_1(w,z; u,v)E_1(w,u)] E_2(v,z) E_2(w,u) + $$ $$
[L_2(w,z; u,v)E_2(w,u)] E_1(w,z)E_1(v,u).
\end{equation}

If the above identities are true for large values of $w,z,u,v$, then their validity extends to $w,z,u,v \in [\C \cup \{\infty\}]  \setminus [K_1 \cup K_2].$
Identity (\ref{merging-5}) simply represents the product rule for the double finite difference $L(w,z; u,v)E(w,u)$. 
From these simple computations we derive the monotonicity of the quotient kernels appearing in the last identity.

\begin{proposition} Let $0 \leq g \leq f \leq 1$ be measurable functions with compact support, on the complex plane.
Then
$$ \frac{L_g (w,z; u,v)}{E_g(v,z)}  \leq  \frac{L_f (w,z; u,v)}{E_f(v,z)} $$
in the sense of kernels, both defined on the complement of the support of $f$.
\end{proposition} 

\begin{proof} We exploit the identity (\ref{merging-4}), with $L_1 = L_g$ and $L_2 = L_{f-g}$. Then $L = L_f$.
Note that the second term in (\ref{merging-4}) is positive definite as a kernel of $(w,z), (u,v)$ by Schur Product Theorem. 
\end{proof}

In particular, we infer from this proposition that, even if the domains overlap, the quotient kernel above remains positive definite.
To this aim we take now $L_1 = L_g, L_2 = L_f$ and $L = L_{f+g}$.

\begin{corollary} Let $g,f: \C \longrightarrow [0,1]$ be measurable functions of compact support. The kernel
$$  \frac{L_{f+g} (w,z; u,v)}{E_{f+g}(v,z)}, \ \ w,z,u,v \notin {\rm supp}(f) \cup {\rm supp}(g),$$
is positive definite.
\end{corollary}

Of immediate relevance are two specializations of the kernel $L$ at points at infinity:
$$ M(w,z) = [v \overline{u} L(w,z; u,v)]_{v,u = \infty} = \langle (T^\ast - \overline{z})^{-1} \xi, (T^\ast - \overline{w})^{-1} \xi \rangle,$$
and
$$ N(w,u) = [v \overline{z}  L(w,z; u,v)]_{v,z = \infty} = \langle (T-w)^{-1} \xi, (T-u)^{-1} \xi \rangle.$$
They represent Gram matrices of the resolvents of $T^\ast$, respectively $T$. Starting from the definition of the kernel $L$, elementary algebra yields:
$$ E(w,z) = 1 - M(w,z),$$
and
$$ \frac{1}{E(w,z)} = 1 + N(w,z).$$
The second identity is also derivable from the determinantal definition of the exponential kernel:
$$ E(w,z) = \det [ (T-w)(T^\ast - \overline{z})(T-w)^{-1} (T^\ast - \overline{z})^{-1}],$$
and obviously
$$ \frac{1}{E(w,z)} = \det [(T^\ast - \overline{z})(T-w)(T^\ast - \overline{z})^{-1}(T-w)^{-1}].$$
See \cite{Gustafsson-Putinar-LNM} for details.

Assume as before that the thick compact $K$ is the disjoint (in the sense of Lebesgue measure) union of $K_1$ and $K_2$. Formula (\ref{merging-1}) becomes
$$ 1 + M(w,z) = (1+M_1(w,z))(1+M_2(w,z)), \ \ w,z \notin K.$$
Let $T_j \in {\mathcal L}(H_j)$ denote the associated hyponormal operators. To simplify notation write $\eta(z) = (T - {z})^{-1}\xi$, $\eta_j(z) = (T_j - {z})^{-1}\xi_j, \ j =1,2.$
We find:
$$ \langle \eta(z), \eta(w) \rangle =  \langle \eta_1(z), \eta_1(w) \rangle +  \langle \eta_2(z), \eta_2(w) \rangle + $$ $$
\langle \eta_1(z) \otimes \eta_2(z), \eta_1(w) \otimes \eta_2(w) \rangle, \ \  z \notin K.$$
In other terms the Gram kernels of the family of vectors $\eta(z) \in H$ and \\
$(\eta_1(z), \eta_2(z),  \eta_1(z) \otimes \eta_2(z)) \in H_1 \oplus H_2 \oplus H_1 \overline{\otimes} H_2,$
coincide. We denote by $\overline{\otimes}$ the Hilbertian tensor product. Consequently there exists an isometric linear transform
$$ V: H \longrightarrow H_1 \oplus H_2 \oplus H_1 \overline{\otimes} H_2,$$ satisfying
$$ V \eta(z) = (\eta_1(z), \eta_2(z),  \eta_1(z) \otimes \eta_2(z)).$$ In particular, equating the first order term at infinity: $V\xi = (\xi_1, \xi_2, 0).$
But $T \eta(z) = \xi + z \eta(z),$ and similarly for $T_j, j =1,2.$ In addition, multiplication by $z$ commutes with the operator $V$, to the effect that:
$$ V (T \eta(z) -\xi) = (T_1 \eta_1(z) - \xi_1, T_2 \eta_2(z) - \xi_2, (T_1 \eta_1(z) - \xi_1) \otimes \eta_2(z)) = $$ $$
(T_1 \eta_1(z) - \xi_1, T_2 \eta_2(z) - \xi_2, \eta_1(z) \otimes (T_2 \eta_2(z)- \xi_2)).$$
The free terms cancel:
$$ V T \eta(z)  = (T_1 \eta_1(z), T_2 \eta_2(z), (T_1 \eta_1(z) - \xi_1) \otimes \eta_2(z)) = $$ $$
(T_1 \eta_1(z), T_2 \eta_2(z), \eta_1(z) \otimes (T_2 \eta_2(z)- \xi_2)).$$
Since the vector $\xi$ is rationally $T$-cyclic,  the vectors $\eta(z), z \notin K$ span the Hilbert space $H$. In conclusion we have proved the following description of the hyponormal quantization of the merging island.

\begin{theorem} Let $K_1$ and $K_2$ be two (area measure) disjoint, thick compact sets contained in the complex plane. Let $T \in {\mathcal L}(H)$ denote the  hyponormal operator 
associated to $K_1 \cup K_2$, and $T_j \in {\mathcal L}(H_j)$ the operators corresponding to $K_J, j=1,2.$ There exists a linear isometric transform
$V: H \longrightarrow H_1 \oplus H_2 \oplus H_1 \overline{\otimes} H_2$, satisfying
\begin{equation} 
 V T =  \left( \begin{array}{ccc}
         T_1 & 0 & 0 \\
         0 & T_2 & 0 \\
         0 & - \xi_1 \otimes I & T_1 \otimes I\\
         \end{array} \right ) V =  \left( \begin{array}{ccc}
         T_1 & 0 & 0 \\
         0 & T_2 & 0 \\
          - I \otimes \xi_2& 0  & I \otimes T_2
         \end{array} \right ) V.
\end{equation}
\end{theorem}
Towards finding the structure of the the finite central truncations along the $T$-cyclic subspaces generated by $\xi$ we note:

\begin{corollary} In the condition of the Theorem, assume $f(z)$ is an analytic function defined in a neighborhood of the compact set $K = K_1 \cup K_2$. Then Riesz-Dunford functional calculus reads:
\begin{equation}
 V f(T) =  \left( \begin{array}{ccc}
         f(T_1) & 0 & 0 \\
         0 & f(T_2) & 0 \\
         0 & A_{32} & f(T_1) \otimes I\\
         \end{array} \right ) V =  \left( \begin{array}{ccc}
         f(T_1) & 0 & 0 \\
         0 & f(T_2) & 0 \\
          A_{31} & 0  & I \otimes f(T_2)
         \end{array} \right ) V.
\end{equation}
\end{corollary} 

 The finite difference convention
$ f[u,v] = \frac{f(u) - f(v)}{u-v}$ yields:
$$ A_{32} = - f[T_1\otimes I, I \otimes T_2] \xi_1 \otimes I,$$
respectively
$$ A_{31} = - f[T_1\otimes I, I \otimes T_2] I \otimes \xi_2.$$
A more symmetric expression is obtained after evaluating such a function on the cyclic vector:
\begin{equation}
 V f(T) \xi = ( f(T_1)\xi_1, f(T_2)\xi_2, - f[T_1\otimes I, I \otimes T_2]\xi_1 \otimes \xi_2).
 \end{equation}
As a verification, the resolvent identity
$$ \frac{(u-z)^{-1} -(v-z)^{-1} }{u-v} = - (u-z)^{-1}(v-z)^{-1}$$
implies the defining relation of the isometry $V$:
$$V (T-z)^{-1}\xi = ((T_1-z)^{-1}\xi_1, (T_2-z)^{-1}\xi_2, (T_1-z)^{-1}\xi_1 \otimes (T_2-z)^{-1}\xi_2), \ \ z \notin K_1 \cup K_2.$$

\section{Packing quadrature domains}

Studying unions of quadrature domains is simplified by the rationality of the associated exponential transform and four argument kernel $L$.
Speficially, let $\Omega$ be a (non necessarily connected) quadrature domain with moment data finitely encoded by the exponential transform: 
$$ Q(w, z) = P(w) \overline{P(z)} E(w,z), \ \ w,z \notin \overline{\Omega},\ \ z,w \notin \overline{\Omega}.$$
 We recall that $P(w)$ is the minimal polynomial vanishing at the quadrature nodes and $Q(w,z)$ is a hermitian kernel, polynomial (analytic/anti-analytic) in both variables.
 Consequently, the $L$ kernel turns out to be rational, too:
$$ P(v) \overline{P(z)} L(w,z; u,v) = \frac{ Q(v,z) Q(w,u) - Q(w,z) Q(v,u)}{(v-w) (\overline{u}-\overline{z}) Q(w,u)}, \ \ w,z,u,v \notin  \overline{\Omega}.$$

Taking into account the two diagonal block decomposition of the corresponding matrix model we obtain the following eigenfunction type decomposition
derived from (\ref{double-resolvent}).
 
 \begin{proposition}
 Let $\Omega \subset \C$ be a quadrature domain and let $T$ denote the associated hyponormal operator with rank-one self-commutator and matrix decomposition
 (\ref{weighted-shift}). The corresponding $L$ kernel is
 \begin{equation}
 L(w,z;u,v) = \sum_{k=0}^\infty \langle f_k(w,z), f_k(u,v)\rangle, \ \ |z|,|w|,|u|,|v| > \| T\|,
 \end{equation}
 where the functions $f_k$ are recursively defined by (\ref{direct-summand}).
 \end{proposition}
 
 \subsection{Non overlapping criteria}
 
Theorem \ref{archipelago} specializes to the following certificate.

\begin{theorem} Let $\Omega_j, j \in J,$ be a finite collection of bounded quadrature domains contained in the complex plane. Denote by $Q_j(w,z)$ the corresponding defining polynomials, and let
$\Omega = \cup_{j \in J} \Omega_j,$ respectively \\
$Q(w,z) = \prod_{j \in J} Q_j(w,z), \ \ w,z \notin \overline{\Omega}$. 

The open set $\Omega$ is a quadrature domain if and only if there exists $R > \sup_{w \in \Omega} |w|$, such that the kernel
\begin{equation}\label{QD-kernel}
 K(w,z; u,v) =  \frac{ Q(v,z) Q(w,u) - Q(w,z) Q(v,u)}{(v-w) (\overline{u}-\overline{z}) Q(w,u)}
 \end{equation}
satisfies the positivity conditions (\ref{bounded}) for $|w|, |z|, |u|, |v| > R.$
\end{theorem}

For the proof we simply remark that multiplying the kernel $L$ by the factor  $P(v) \overline{P(z)}$ does not alter the proof and conclusion of Theorem \ref{archipelago}.

A second characterization of quadrature domains non-overlapping can be derived from the structure of the block matrix staircase decomposition of the operator $T$. 

\begin{theorem}\label{matrix-chain}
Let $\Omega_j, j \in J,$ be a finite collection of bounded quadrature domains contained in the complex plane. Denote by $P_j(z)$ the corresponding minimal polynomials vanishing at the quadrature nodes, and let $P = \prod_{j \in J} P_j.$
Then $\Omega$ is a quadrature domain if and only if:

a) The joint exponential transform $E(w,z) = \prod_{j \in J} E_j(w,z)$  admits a sums of squares decomposition of the form:
$$ P(w) \overline{P(z)} (1-E(w,z)) = \sum_{k=0}^{d-1} Q_k(w) \overline{Q_k(z)}, \ \ |w|, |z| > \sup_{z \in \Omega} |z|,$$
where $d = \deg P$ and $Q_k$ are polynomials of the exact degree $k$: $\deg Q_k = k, \ \ 0 \leq k < d,$

b). The recursively generated $d \times d$ matrices (\ref{recurrence}) satisfy: $A_k > 0, \ k \geq 0,$ and  $\sup_k \| D_k \| < \infty.$
\end{theorem}

\begin{proof} Condition a) simply means that there exists a $d\times d$ matrix $D_0^\ast$ with cyclic vector $\xi$, satisfying
$$ 1- E(w,z) = \langle (D^\ast - \overline{z})^{-1} \xi , (D^\ast - \overline{w})^{-1} \xi \rangle, \ \  |w|, |z| > \sup_{z \in \Omega} |z|.$$
Cf. Proposition 4.2 in \cite{Gustafsson-Putinar-LNM}.

Condition b) means that the recurrence scheme (\ref{recurrence}) runs to infinity and produces a linear bounded operator possessing the block decomposition
(\ref{weighted-shift}). Then $T$ is a hyponormal operator with rank-one self-commutator with principal function $g_T = \chi_\Omega$ (consequence of the exponential transform
product structure). In addition the linear span ${\rm span} \{ T^{\ast^k} \xi, \ k \geq 0 \} = H_0$ is finite dimensional, hence $\Omega$ is a quadrature domain.
\end{proof}

For the kernel $K$ attached to a quadrature domain $\Omega$ one can write 
$$ K(w,z; u,v) =  \frac{ P(w,z; u,v)}{Q(w,u)}, $$
where $P$ is a polynomial function of degree less than or equal to $d-1$ in each variable. We derive from here a finite step Pad\'e type reconstruction algorithm for the entire kernel $K$, and implicitly $L$.
In view of the block matrix staircase decomposition of the operator $T$ we adapt the Neumann series summands as follows:
$$ g_n(w,z) = \overline{p(z)} f_n(w,z), \ \ n \geq 0,$$
so that the recurrence relations (\ref{direct-summand}) become:
$$  g_{n+1}(w,z) = (D_{n+1}-w)^{-1} A_n g_n(w,z), \ n\geq 0,$$
$$ g_0(w,z) = (D_0-w)^{-1}  \overline{p(z)} (D_0^\ast - \overline{z})^{-1} \xi.$$
The minimal polynomial of the matrix $D_0$ is $p$, therefore $ \overline{p(z)} (D_0^\ast - \overline{z})^{-1} $ is an $H_0$-valued polynomial function of degree less than or equal to $d-1$.
This property extends then to all rational functions $g_n$, written now in a Taylor expansion at infinity:
$$g_n(w,z) = \frac{ c_{n+1}(z)}{w^{n+1}} +  \frac{ c_{n+2}(z)}{w^{n+2}} + \ldots, \ \ n \geq 0.$$
where $z \mapsto c_\ell(z), \ \ell \geq n+1,$ are antianalytic polynomials of degree at most $d-1$.
The inherited Neumann series expansion is still convergent uniformly on compact subsets of the complement of the disk of radius $\| T \|$:
$$ K(w,z; u,v) =  \sum_{k=0}^\infty \langle g_k(w,z), g_k(u,v)\rangle.$$
Consequently, the following identity of Laurent series holds:
$$ P(w,z; u,v) = Q(w,u) K(w,z; u,v) =  \sum_{k=0}^\infty Q(w,u) \langle g_k(w,z), g_k(u,v)\rangle.$$
Chasing the negative degrees with respect to the variables $w$ and $u$ we reach the following conclusion.

\begin{proposition} Let $\Omega \subset \C$ be a bounded quadrature domain of order $d \geq 1$, with associated four variables kernel $K = \frac{P}{Q}$ (\ref{QD-kernel}). The polynomial numerator $P$ of 
$K$ is the non-negative part of the Laurent series 
$$ P(w,z; u,v) =  \sum_{k=0}^{d-1} Q(w,u) \langle g_k(w,z), g_k(u,v)\rangle + O(\frac{1}{w}, \frac{1}{\overline{u}}).$$
\end{proposition}

\subsection{Resolvent realization of the four argument kernel}

We adapt below the ubiquitous ``lurking isometry technique" appearing in bounded analytic interpolation and control theory, \cite{Agler-McCarthy-Young}.

The familiar polarization formula implies that every hermitian monomial in $d$ complex variables can be written as a difference of two hermitian squares:
$$ 2[\tau^\alpha \overline{\sigma}^\beta + \sigma^\beta \overline{\tau}^\alpha] = |\tau^\alpha + \sigma^\beta|^2 - |\tau^\alpha - \sigma^\beta|^2.$$
Consequently we can decompose any hermitian polynomial $P$ as a difference of sums of hermitian squares:
$$ P(\tau,\overline{\sigma}) = \langle p_1(\tau), p_1(\sigma) \rangle - \langle p_2(\tau), p_2(\sigma)\rangle.$$
Above $p_j : \C^d  \longrightarrow K_j,$ are polynomial maps with values in finite dimensional Hilbert spaces $K_j,  \ j = 1,2,$.

\begin{proposition} Let $U \subset \C^d$ be a connected open set, and let $L : U \times U^\ast \longrightarrow \C$ be a hermitian kernel.
Assume that $L$ is rational of the following form:
$$ L(\tau, \overline{\sigma}) = \frac{ \langle p_1(\tau), p_1(\sigma) \rangle - \langle p_2(\tau), p_2(\sigma)\rangle}{q(\tau)\overline{q(\sigma)} - \langle r(\tau), r(\sigma)\rangle},$$
where $P(\tau,\overline{\sigma})$ is a hermitian polynomial, $q(\tau)$ is a polynomial, while $r(\tau) \in K_3$  is a polynomial map with values in a finite dimensional Hilbert space $K_3$,
such that $|q(\tau)| > \| r(\tau)\|, \tau \in U.$

The kernel $L$ is positive semi-definite on $U \times U^\ast$ if and only if there exists an auxiliary Hilbert space $K$ and a linear isometric transform:
\begin{equation}\label{block}
 \left( \begin{array}{cc}
              A & B\\
               \\
              C & D
              \end{array} \right) : \left(\begin{array}{c}
              K_3 \otimes K \\
              \oplus\\
              K_1 \end{array} \right) \longrightarrow \left( \begin{array}{c}
              K \\
              \oplus\\
              K_2 \end{array} \right)  
\end{equation}
with the property
\begin{equation}
 L(\tau, \overline{\sigma}) = 
 $$ $$ \langle [q(\tau)I - A (r(\tau) \otimes I)]^{-1} B p_1(\tau), [q(\sigma)I - A (r(\sigma) \otimes I)]^{-1} B p_1(\sigma)\rangle , \ \ \tau, \sigma \in U.
\end{equation}
Moreover, 
$$ p_2(\tau) = C [ r(\tau) \otimes (q(\tau)I - A (r(\tau) \otimes I))^{-1} B p_1(\tau)] + D p_1(\tau), \ \tau \in U.$$
\end{proposition} 

The above tensor product is the completed Hilbert space tensor product in case the dimension of $K$ is infinite.
If $\dim K < \infty$, in order for such an isometry to exist, we may enlarge the space $K_2$ by a direct summand, so that the dimension of 
the domain is not greater than the dimension of the codomain. This operation is not necessary in the quadrature domains framework, since then $\dim K = \infty$ anyway:
$K$ is the Hilbert space carrying the associated hyponormal operator with rank-one self-commutator.

\begin{proof}
In virtue of Kolmogorov lemma, there exists a finite dimensional Hilbert space $K$ and vectors $x_\tau \in K$, such that
$$ L(\tau, \overline{\sigma}) = \langle x_\tau, x_\sigma \rangle, \ \ \tau, \sigma \in U.$$
For $\tau, \sigma \in U$ one has:
$$ \langle p_1(\tau), p_1(\sigma) \rangle - \langle p_2(\tau), p_2(\sigma)\rangle =  \langle x_\tau, x_\sigma \rangle (q(\tau)\overline{q(\sigma)} - \langle r(\tau), r(\sigma)\rangle),$$
or equivalently
$$\langle q(\tau) x_\tau, q(\sigma) x_\sigma \rangle + \langle p_2(\tau), p_2(\sigma)\rangle = \langle r(\tau) \otimes x_\tau, r(\sigma) \otimes x_\sigma \rangle +  \langle p_1(\tau), p_1(\sigma) \rangle.$$
In other terms, the Gram matrices of the family of vectors indexed by $\tau \in F$:
$$ (q(\tau) x_\tau, p_2(\tau)) \in K \oplus K_2,$$
and
$$ (r(\tau) \otimes x_\tau, p_1(\tau)) \in (K_3 \otimes K) \oplus K_1,$$
are identical.

Consequently there exists a partial isometry 
mapping one system of vectors to another. The above linear map is independent of $\tau$ and it is isometric on the linear span of the vectors  $(r(\tau) \otimes x_\tau, p_1(\tau)), \tau \in U.$
The only obstruction to extend it to a full isometry is the possible dimension difference between $(K_3 \otimes K) \oplus K_1$ and $K \oplus K_2$. As already mentioned, we can enlarge the Hilbert space $K_2$ by a direct summand, so that $\dim [(K_3 \otimes K) \oplus K_1] \leq \dim [K\oplus K_2]$. In this way we can assume that the block operator (\ref{block})  is an isometry.

We find:
$$ q(\tau) x_\tau = A ( r(\tau) \otimes x_\tau) + B p_1(\tau), \ \ \tau \in F, $$
and
$$ p_2(\tau) = C ( r(\tau) \otimes x_\tau) + D p_1(\tau), \ \tau \in F.$$
From the first equation we infer
$$ x_\tau = [q(\tau)I - A (r(\tau) \otimes I)]^{-1} B p_1(\tau).$$
Note that the operator $q(\tau)I  - A (r(\tau)\otimes I : K_3 \longrightarrow K_3$ is invertible for all $\tau \in U$ because
$|q(\tau)| > \| r(\tau) \| \geq \| A (r(\tau) \otimes I \|, \ \tau \in U,$ by assumption and the contractivity of $A$.

We define the analytic function $$X(\tau) = [q(\tau)I - A (r(\tau) \otimes I)]^{-1} B p_1(\tau) \in K, \tau \in U.$$
In particular
$$ q(\tau) X(\tau) = A ( r(\tau) \otimes X(\tau)) + B p_1(\tau), \ \ \tau \in U, $$
so that
$$  p_2(\tau) = C ( r(\tau) \otimes X(\tau)) + D p_1(\tau), \ \tau \in U,$$
as desired.

Conversely, if the isometry (\ref{block}) exists, we define $X(\tau)$ as above and note that the Gram matrices of the systems of vectors  $(q(\tau) X(\tau), p_2(\tau))$
and $(r(\tau) \otimes X(\tau), p_1(\tau))$ coincide on $U$. Hence 
$$ \langle p_1(\tau), p_1(\sigma) \rangle - \langle p_2(\tau), p_2(\sigma)\rangle =  \langle X(\tau), X(\sigma) \rangle (q(\tau)\overline{q(\sigma)} - \langle r(\tau), r(\sigma)\rangle),$$
for every $(\tau,\sigma) \in U \times U$. We infer that the kernel $L$ is positive semi-definite.
\end{proof} 

The advantages of the Hilbert space factorization implemented by such a linear system governed by an isometric transform were recognized early in linear control theory. Later on, classical results of bounded analytic interpolation, such as the Nevanlinna-Pick Theorem, or Carath\'eodory-Fej\'er Theorem, were enhanced by this matrix analysis approach. We refer to \cite{Agler-McCarthy-Young} for details. 

To give the quintessential example of Hilbert space factorization, consider Szeg\"o's type kernel
$$ L(\tau, \sigma) = \frac{1}{\tau \overline{\sigma} -1}, \ \ |\tau|, |\sigma| >1.$$
Let $S e_n = e_{n+1}, \ \ n \geq 0,$ be the unilateral shift, acting on a Hilbert space with orthonormal basis $(e_n)_{n=0}^\infty$.
The Neumann series expansion leads to the factorization:
$$  \frac{1}{\tau \overline{\sigma} -1} = \sum_{k=0}^\infty \frac{1}{\tau^{k+1} \overline{\sigma}^{k+1}} = $$ $$ 
\sum_{k,\ell =0}^\infty \frac{ \langle e_k ,  e_\ell \rangle}
{ \tau^{k+1} \overline{\sigma}^{\ell+1}} = \langle (S-\tau)^{-1} e_0,  (S-\sigma)^{-1} e_0 \rangle, \ \ |\tau|, |\sigma| >1.$$
The isometry $V: H \oplus \C \longrightarrow H$ defined as
$$ V( f, a) = Sf - a e_0, \ \ f \in H, a \in \C,$$ 
implements the identity stated above. Indeed,
$$ V ( (S-\tau)^{-1} e_0, 1) = S (S-\tau)^{-1} e_0 - e_0 = \tau (S-\tau)^{-1} e_0.$$
In the notation of the Proposition, $A=S$, $B = - e_0 \otimes I, C = 0, D = 0.$ Moreover $q(\tau) = \tau, p_1(\tau) = 1, p_2 = 0.$ The verification
is straightforward:
$$ X(\tau) = (\tau   - S)^{-1} (-e_0).$$

Minor modifications in the proof of the Proposition allow a relaxation of the statement, from a scalar valued numerator to a matrix valued one. Specifically, one can assume $p_1(\tau), p_2(\tau)$ to be polynomial functions with values $n \times n$ matrices. Then the numerator becomes $p_1(\sigma)^\ast p_1(\tau) - p_2(\sigma)^\ast p_2(\tau).$ This observation is relevant for the analysis of the $L$-kernel associated to a quadrature domain.

Leaving aside the positivity of the kernel $L$ associated to a quadrature domains, we mention the existence of an array of realization formulae of $L$ as a ``compressed resolvent" of a pencil of finite matrices. Such representations of rational functions depending on several variables are rooted in the theory of formal languages and control theory. We refer to Table 1 in \cite{Stefan-Welters} for an overview of the theory.

 \section{Two disks} The rest of the present article deals with the example of two disks.
 Starting with the merging formula (\ref{merging-5}) we investigate the kernel $L$ associated to the union $U$ of two disks
$\D(a,r)$ and $\D(b,s)$. A short computation yields:
$$ L_U(w,z; u,v) E_U(w,u) = $$ $$\frac{ s^2((v-a)(\overline{z}- \overline{a}) - r^2)((w-a)(\overline{u}- \overline{a}) - r^2) + r^2((w-b)(\overline{z}- \overline{b}) - s^2)((v-b)(\overline{u}- \overline{b}) - s^2)}
{(w-a)(\overline{z}- \overline{a})(v-a)(\overline{u}- \overline{a}) (w-b)(\overline{z}- \overline{b})(v-b)(\overline{u}- \overline{b}) }.$$

One can specialize to a pair of symmetric disks: $a = -b >0$ and $s=r = 1$. Then
$$ L_U(w,z; u,v) E_U(w,u) =  $$ $$\frac{ 2 w\overline{z} v \overline{u}  + 2 a^2(vw+ \overline{u} \overline{z} + (w+ v)(\overline{z} + \overline{u})) - (w+ v)(\overline{z} + \overline{u}) + 2 (a^2-1)}
{(w^2- a^2)(\overline{z}^2-a^2) (v^2-a^2) (\overline{u}^2-a^2)}.$$
On the other hand
$$ E_U(w,u) = [1-\frac{1}{(w-a)(\overline{u}-a)}][1-\frac{1}{(w+a)(\overline{u}+a)}] = \frac{(w\overline{u} + a^2 -1)^2 - a^2 (w + \overline{u})^2}{(w^2- a^2)(\overline{u}^2-a^2)}.$$
Consequently
$$ L_U(w,z; u,v)  =  $$ $$\frac{ 2 w\overline{z} v \overline{u}  + 2 a^2(vw+ \overline{u} \overline{z} + (w+ v)(\overline{z} + \overline{u})) - (w+ v)(\overline{z} + \overline{u}) + 2 (a^2-1)^2}
{(\overline{z}^2-a^2) (v^2-a^2)((w\overline{u} + a^2 -1)^2 - a^2 (w + \overline{u})^2)}.$$

We run  a first iteration of the recurrence scheme appearing in Theorem \ref{matrix-chain}. First, the minimal polynomial is $p(z) = z^2 - a^2$ and the joint exponential transform
satisfies
$$ |z^2-a^2|^2 (1 - E(w,z)) = |z^2-a^2|^2 - (|z-a|^2 -1)(|z+a|^2-1) = 2|z|^2 + 2a^2 -1.$$
A first condition, derived from assertion a) in the theorem is: $2a^2 -1 > 0$, that is $$a > \frac{1}{\sqrt{2}}.$$ 
This is of course not sufficient for the two disks to be disjoint.

Assume $a > \frac{1}{\sqrt{2}}$. We seek a $2\times 2$ matrix $D_0^\ast$, with cyclic vector $\xi$ and spectrum equal to $\{ -a, a\}$, subject to the identification:
$$ |z^2-a^2|^2 \| (D_0^\ast - \overline{z})^{-1} \xi \|^2 = 2|z|^2 + 2a^2 -1.$$
In virtue of Cayley-Hamilton Theorem, the minimal polynomial of this matrix is $z^2-a^2$, so that
$$ \| (D_0^\ast + \overline{z}) \xi \|^2 = 2|z|^2 + 2a^2 -1.$$
By identifying coefficients we infer:
$$ \| \xi \|^2 = 2, \  \langle D_0^\ast \xi, \xi \rangle = 0, \ \ \|D_0^\ast \xi\|^2 = 2a^2-1.$$
We can choose $\xi = (\sqrt{2}, 0)^T$ and consequently
$$ D_0^\ast = \left( \begin{array}{cc}
                       0 & \frac{a^2}{\sqrt{a^2 - \frac{1}{2}}}\\
                       \sqrt{a^2 - \frac{1}{2}} & 0 \\
                       \end{array} \right).$$
According to condition b) in Theorem \ref{matrix-chain}, the matrix $- [D_0^\ast , D_0] + \xi \langle \cdot, \xi \rangle$ must be positive definite. That is, 
$$     \left( \begin{array}{cc}
                       -\frac{a^4}{{a^2 - \frac{1}{2}}}  + (a^2 - \frac{1}{2}) + 2 & 0\\
                       0 & -a^2 + \frac{1}{2} +  \frac{a^4}{{a^2 - \frac{1}{2}}}\\
                       \end{array} \right) > 0.$$                
The conditions for the two diagonal entries to be positive are  $ a > \frac{1}{2}$ and  
$$a > \frac{\sqrt{3}}{2}.$$

Assume $a > \frac{\sqrt{3}}{2}$. We define the positive matrix
$$ A_0 =   \left( \begin{array}{cc}
                       \sqrt{\frac{a^2-\frac{3}{4}}{a^2-\frac{1}{2}}} & 0\\
                      0 & \sqrt{\frac{a^2-\frac{1}{4}}{a^2-\frac{1}{2}}} \\
                       \end{array} \right)$$ 
                       and continue
                       $$ D_1 = A_0^{-1} D_0 A_0 = \left( \begin{array}{cc}
                       0 & \sqrt{\frac{ (a^2-\frac{1}{2}) (a^2 - \frac{1}{4})} { a^2 - \frac{3}{4}}}\\
                      a^2  \sqrt{\frac{a^2-\frac{3}{4}}{ (a^2-\frac{1}{2})(a^2- \frac{1}{4})}}& 0 \\
                       \end{array} \right).$$ 
                       Next we have to assure that the matrix $A_0^2 - [D_1^\ast, D_1]$ is positive. Fortunately this continues to be a diagonal matrix:
                       $$ \left( \begin{array}{cc}
                       \frac{a^2 - \frac{3}{4}}{a^2 - \frac{1}{2}} - \frac{ a^4 (a^2 - \frac{3}{4})}{  (a^2-\frac{1}{2})(a^2- \frac{1}{4})} +  \frac{ (a^2-\frac{1}{2})(a^2- \frac{1}{4})}{a^2 - \frac{3}{4}}& 0\\
                     0 & \frac{a^2-\frac{1}{4}}{a^2 - \frac{1}{2}} + \frac{ a^4 (a^2 - \frac{3}{4})}{  (a^2-\frac{1}{2})(a^2- \frac{1}{4})} -  \frac{ (a^2-\frac{1}{2})(a^2- \frac{1}{4})}{a^2 - \frac{3}{4}}\\
                       \end{array} \right).$$ If positive, this matrix will be $A_1^2$. The substitution $t = a^2 - \frac{1}{2}$ and some calculation on the second diagonal entry 
                       implies
                       $$a^2 > \frac{1}{2} (1 + \frac{1}{\sqrt{2}}).$$
                      In view of Theorem \ref{matrix-chain}, after continuing the process, the sequence of improved lower bounds for $a^2$: $ \frac{1}{2} < \frac{3}{4} <  \frac{1}{2} (1 + \frac{1}{\sqrt{2}}) \ldots$ will necessarily reach the correct value $a^2 \geq 1.$


\section{Gravi-equivalent deformations of overlapping disks}

In this section we give examples of phenomena, related to positive definiteness, occurring when the density function $g:\C\to\R$
is allowed to take values greater than one, but only integer values. This leads to the realm of multi-sheeted algebraic domains,
or quadrature Riemann surfaces.

\subsection{Evolution by level lines}\label{sec:level lines}

As observed in Example~2 in \cite{Gustafsson-Putinar-2005}, and further studied in \cite{Tkachev-2006},
the exponential kernel $1-E_g (z,w)$  for the density $g=\chi_{\D(a_1,r_1)}+\chi_{\D(a_2,r_2)}$  
is positive semidefinite if and only if $r_1^2+r_2^2\leq |a_1-a_2|^2$. 
In the case of equality the two boundary circles intersect under right 
angles, and this fact gives a clear picture of how far one circle can penetrate other circle
while keeping positive semidefiniteness. In particular,
one circle cannot be completely swallowed by another circle.
It may therefore be of interest to see an example of a smooth algebraic domain 
completely contained in another with the positive semi-definiteness preserved.

Such an example can be produced by using level lines of the defining polynomial of two
intersecting circles as above.
We shall concentrate on the special choice $a_j=\pm 1$, $r_j=\sqrt{2}$ ($j=1,2$), for which there are
some nice features when turning to the spherical metric. Let thus
$$
g=\chi_{\D(-1,\sqrt{2})}+\chi_{\D(1,\sqrt{2})}.
$$
It is clear from the mean-value property of harmonic functions that the ``quadrature'' property
\begin{equation}\label{intK}
\int_K  h(z) g(z)dxdy= 2\pi\big(h(-1)+h(1)\big)
\end{equation}
holds for all functions $h$ harmonic in a neighborhood of  $K={\rm supp\,}g$. This relation can be extended to saying that the inequality $\geq $
holds for subharmonic functions $h$, and that inequality is important in certain contexts (for example uniqueness questions).

The two circles, and lines of discontinuity for $g$, coincide with the zero locus of the real polynomial
$$
Q(z,{z})=\big(|z-1|^2-2\big)\big(|z+1|^2-2\big)
$$
$$
=z^2\bar{z}^2-4z\bar{z}-z^2-\bar{z}^2+1=R(z,\bar{z}).
$$
Here we have on the last line have related $Q(z,z)$ to the complex analytic polynomial $R(z,w)$
introduced in Section~\ref{sec:intro}.

It will be good to notice that $Q(z,{z})$ has critical points at $z=0$ (local maximum with $Q=1$) and $z=\pm \sqrt{3}$
(local minima with $Q=-8$). For convenience, values of $g$ and signs of $Q$ are indicated in Figure~\ref{fig:circles}.
It also good to notice that
\begin{equation}\label{EQP}
E_g(z,z)=E_{\D(-1,\sqrt{2})}(z,z) E_{\D(1,\sqrt{2})}(z,z)=\frac{Q(z,{z})}{|z^2-1|^2} \quad(|z|>>1).
\end{equation}  


\begin{figure}
\begin{center}
\includegraphics[width=\linewidth]{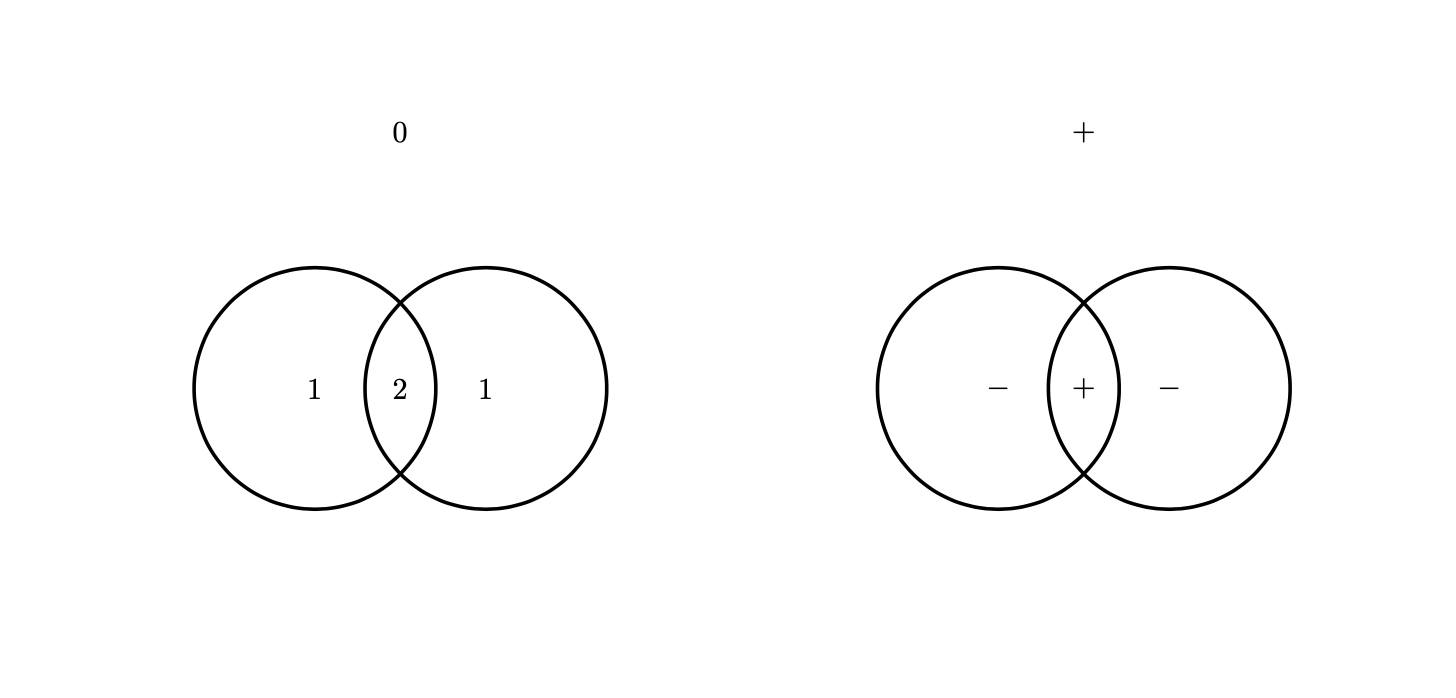} 
\end{center}
\caption{Left: values of $g(z)$. 
Right:  signs of $Q(z,{z})$.}
\label{fig:circles}
\end{figure} 


Below we shall consider more general level lines of $Q(z,{z})$, and to this purpose 
we consider, as before, the sub-level sets
$$
\Omega(t)=\{z\in\C; Q(z,{z})<t\}.
$$
All real values of $t$ are of interest here, but we shall mainly focus on those in the interval $0\leq t\leq 1$.
The boundary $\Gamma(t)=\partial\Omega(t)$ is a plane algebraic curve, and it has an associated Riemann surface, which
can be identified with  the complex algebraic curve $R(z,w)=t$ after completion and desingularization of the latter.
This Riemann surface is ``real'' in the sense that it has an anticonformal involution, namely 
$(z,w)\mapsto (\bar{w},\bar{z})$, and the real locus $\Gamma(t)$ of the curve is essentially the fixed 
point set of this involution (the locus may also contain some isolated accidental points). 
The Riemann surface has genus one except for some special values of $t$:
for $t=1$ the genus is zero, and for $t=0$ and for $t=-8$ it is disconnected with each component having genus zero.  

As $t$ increases from $t=0$, the two intersecting circles in Figure~\ref{fig:circles} become
two smooth curves, and the disconnected set where $Q(z,{z})<0$ becomes the
doubly connected and smoothly bounded domain $\Omega(t)$, see Figure~\ref{fig:level sets}. 
We define the corresponding density function $g_t$ to be one in $\Omega(t)$, and to be two in the hole
(that is, in the bounded component of $\C\setminus \Omega(t)$):
\begin{equation}\label{rhot}
g_t=
\begin{cases}
1 \quad\text{in }\Omega(t),\\
2\quad\text{in the hole}. 
\end{cases}
\end{equation}
The domain $\Omega(t)$ is actually only part of a multi-sheeted domain, which also covers the hole with two sheets
connected to each other via branch points in the hole. This is to explain (\ref{rhot}).
Accordingly, it will be seen in the proof below that as a result of integration
with respect to the weight $g_t$, the inner boundary component of $\Omega(t)$ is to be oriented as the boundary of the hole.
See specifically equation (\ref{hole}).
 
As $t\nearrow 1$ the hole is gradually filled in and eventually becomes a residual isolated point, the origin  $\{0\}$.
Ignoring that point (which is a ``special point'' in the theory of quadrature domains \cite{Shapiro-1992, Gustafsson-Shapiro-2005})
we have
$$
g_1=\chi_{\Omega(1)},
$$
where $\Omega(1)$ is a well-studied simply connected two-point quadrature domain, a Neumann oval or hippopede.

Setting $P(z)=z^2-1$ we see that
$$
|P(z)|^2-\big(Q(z,{z})-t\big)=4|z|^2+t
$$
is positive semidefinite for all $t\geq 0$.
The left member in this expression is connected to the exponential transform $E_{g_t}(z,{z})$ by
$$
E_{g_t}(z,{z})=\frac{Q(z,{z})-t}{|P(z)|^2} \quad (|z|>>1),
$$
compare (\ref{EQP}) and see further  \cite{Gustafsson-Tkachev-2011}.

It will be of major interest to notice that the integer-valued densities $g_t$, $0\leq t\leq  1,$
all enjoy the same quadrature identity as (\ref{intK}). We state this formally as 

\begin{proposition}
For all $0\leq t\leq 1$, the quadrature identity
\begin{equation}\label{qirho}
\int h(z)g_t(z)dxdy=2\pi\big( h(-1)+h(1))
\end{equation}
holds for functions $h$ which are harmonic over the support of $g_t$. 
\end{proposition}

Such a statement is implicitly contained in \cite{Gustafsson-Tkachev-2011}, with a
proof based on conformal mapping from an abstract uniformizing Riemann surface 
together with a residue calculation. However,  the statements in \cite{Gustafsson-Tkachev-2011} mainly refer to the
spherical metric, and in that case the locations and strengths of the quadrature nodes depend on $t$ and are less explicit.
Therefore it might be in order to give a direct proof in the present case.

\begin{proof}
Since the support of $g_t$ is simply connected it is enough to consider analytic test functions $h$ in the identity (\ref{qirho}).  
Solving the algebraic equation $R(z,\bar{z})=t$ for $\bar{z}$ gives 
\begin{equation}\label{Stz}
\bar{z}=\frac{1}{z^2-1}\Big(2z\pm\sqrt{(z^2+1)^2+t(z^2-1)}\Big),
\end{equation}
where the right member is an expression for the Schwarz function $S_t(z)$ of $\Gamma(t)$.
This means that, selecting the appropriate branch, $S_t(z)$ is analytic in a neighborhood of $\Gamma(t)$
and satisfies $S_t(z)=\bar{z}$ on $\Gamma(t)$.  See \cite{Davis-1974, Shapiro-1992} in general. 

In the present case the Schwarz function is algebraic, 
and it has two poles and four branch points. The branch points are obtained from 
$z^4+(2+t)z^2+1-t=0$, which gives 
$$
z=\pm\frac{1}{\sqrt{2}}\sqrt{-2-t \pm \sqrt{t(t+8)}}.
$$
In the interval $0<t<1$ all four branch points lie on the imaginary axis, say
$$
z=\pm \I A, \pm \I B, \quad 0<A<B.
$$
Two of them (namely $\pm \I A$) are inside the hole mentioned at (\ref{rhot}), the two other are
outside $\Omega(t)$. For $t=0$ the branch points fuse in pairs, to become $z=\pm \I$, the intersection points
for the two circles. For $t=1$ the two branch points inside the hole have fused to become the special 
point $z=0$.

It is possible to extend the analysis to the parameter range $-8<t<0$ with somewhat similar results, but it then
becomes more complicated because one has to treat unbounded multi-sheeted domains, and also involve
the four variable exponential transform. And, as mentioned,
further extensions to all real values of $t$ is also possible, with the additional difficulty that one runs into symmetric 
Riemann surfaces which are doubles of non-orientable surfaces (real Riemann surfaces of non-dividing type,
in a different terminology).
For example, when $t<-8$ then there is no real locus at all of the equation $Q(z,{z})=t$, and the corresponding
Riemann surface can be considered as being the double of a Klein's bottle. For $t>1$ the Riemann surface is the double
of a M\"obius band, and there is only one curve in the real, despite the surface has genus one. For $-8<t<0$, as well as for
$0<t<1$, the real locus consists after all of two curves, which what one expects for a curve of genus one.
See \cite{Gustafsson-Tkachev-2011} for further discussions.

Referring to (\ref{rhot}) the treatment of the left member of   (\ref{qirho}) may be started as follows:
$$
2\I\int h(z) g_t(z)dxdy=\int h(z)g_t(z)d{\bar{z}}dz
$$
$$
=\int_{\partial\Omega(t)}h(z)\bar{z}dz+2\int_{\partial (\rm hole)} h(z)\bar{z}dz
$$
\begin{equation}\label{hole}
=\int_{\partial\Omega(t)_{\rm outer}}h(z)\bar{z}dz+\int_{\partial (\rm hole)} h(z)\bar{z}dz.
\end{equation}
Here ``${\rm hole}$'' refers to the bounded component of $\C\setminus \Omega(t)$, or equivalently the region where $g_t=2$.
The inner boundary of $\partial\Omega(t)$ coincides as a set with the boundary of the hole, but it has the opposite orientation.
This explains the formula  (\ref{hole}), where thus both curves are to be oriented counter-clockwise.

In (\ref{hole}) we wish to insert the Schwarz function and compute everything
by residues. The poles of the Schwarz function, $z=\pm 1$, are outside the hole, but the branch points 
inside the hole complicate matters. They make that branch of $S_t(z)$ which equals $\bar{z}$ on the outer boundary
to continue analytically through the boundary of the hole, then pass through a branch cut $[-\I A, \I A]$ and eventually reach
the boundary of the hole again, then from inside and on the other sheet. 
\begin{figure}
\includegraphics[width=\linewidth]{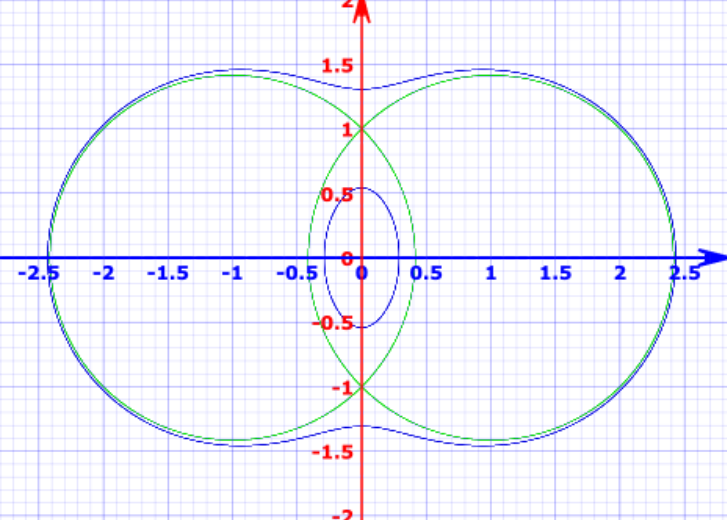}
\caption{The level sets of $Q(z,z)$ corresponding to $t=0, 1/2$.} 
\label{fig:level sets}
\end{figure}

The conclusion is that in the region $\Omega(t)$ there are two single-valued branches of the Schwarz function,
one of which satisfying $S_t(z)=\bar{z}$ on the outer boundary, the other satisfying the same identity on the boundary of the hole.

Inserting the two branches of the Schwarz function, denoted $S_t^\pm(z)$ below,
the computation runs as follows, where the signs of the final square roots are up to interpretation 
(since the variables are complex it does not really make sense to assign a plus or a minus to a square root).
$$
\int_{\partial\Omega(t)_{\rm outer}}h(z)\bar{z}dz+\int_{\partial (\rm hole)} h(z)\bar{z}dz
$$
$$
=\int_{\partial\Omega(t)_{\rm outer}}h(z)S_t^+(z)dz+\int_{\partial (\rm hole)} h(z)S_t^-(z)dz
$$
$$
=\int_{\partial\Omega(t)_{\rm outer}}\frac{h(z)}{z^2-1}\big(2z+\sqrt{(z^2+1)^2+t(z^2-1)}\big)dz
$$
$$
+\int_{\partial({\rm hole})}\frac{h(z)}{z^2-1}\big(2z-\sqrt{(z^2+1)^2+t(z^2-1)}\big)dz
$$
$$
=\Big(\int_{\partial\Omega(t)_{\rm outer}}+\int_{\partial({\rm hole})}\Big)\,\frac{h(z)}{z^2-1}\cdot2z dz
$$
$$
+\Big(\int_{\partial\Omega(t)_{\rm outer}}-\int_{\partial({\rm hole})}\Big)\,\frac{h(z)}{z^2-1}\sqrt{(z^2+1)^2+t(z^2-1)}dz
$$
$$
=2\pi\I \sum_{z=\pm 1}\res \frac{2zh(z)dz}{z^2-1}+0 \,\,\, \text{(there is no residues in the hole) }+
$$
$$
+\int_{\partial{\Omega(t)}}\frac{h(z)}{z^2-1}\sqrt{(z^2+1)^2+t(z^2-1)}dz.
$$
In the last term there are residues again, at $z=\pm 1$. But evaluated at these points we see that the dependence of $t$
disappears, and in addition the square root resolves into $\pm (z^2+1)$. Thus, choosing the right sign we 
can continue as
$$
=2\pi\I \sum_{z=\pm 1}\res \frac{2zh(z)dz}{z^2-1}+
2\pi\I \sum_{z=\pm 1}\res \frac{(z^2+1)h(z)dz}{z^2-1}
$$
$$
=4\pi\I (h(-1)+h(1)),
$$
as desired.
\end{proof} 


\subsection{Evolution by smashing}\label{sec:smashing}

Besides the above evolution from $g_0$ to $g_1$, there is an evolution by ``smashing'' exceeding
mass, or partial balayage in a different terminology. We may denote the densities in this case $\rho_t$,
$0\leq t\leq 1$, where $\rho_0=g_0$ and $\rho_1=g_1$, while $\rho_t\ne g_t$ for $0<t<1$. The
$\rho_t$ are absolutely continuous measures, but the densities are not integer valued when $0<t<1$. Instead
the density $2$ in the overlapping part of the original two disks is squeezed to become $2-t$ in that
part, and the exceeding mass is lumped outside the disks with density one.  The governing rule is that, 
under the above restriction of values of $\rho_t$, the identity
$$
\int h\rho_0 dxdy= \int h \rho_t dxdy
$$
shall hold for functions $h$ harmonic over the support of $\rho_t$. Actually one should also require the
inequality $\leq $ for subharmonic $h$, but in many cases this is automatic.

The theory and construction for this type of process is originally due to M.~Sakai \cite{Sakai-1982}, but
have later been developed by other authors.
The most natural defining property of the process is that the energy of $\rho_t-\rho_0$ is minimized
under the constraint that 
$$
\rho_t\leq
\begin{cases}
2-t &\text{in\, } \D(-1,\sqrt{2})\cap \D(1,\sqrt{2}),\\
1 &\text{outside\, }\D(-1,\sqrt{2})\cap \D(1,\sqrt{2})\end{cases}.
$$
The mentioned (Newtonian) energy becomes a Dirichlet integral (or Sobolev norm)
when formulated in terms of potentials, and this leads to variational inequality formulations, as well as formulations
in terms of obstacle problems, for the construction of $\rho_t$. See \cite{Gustafsson-2004} for an overview. 
Equivalent formulations exist within
the theory of weighted equilibrium distributions as developed by E.~Saff and V.~Totik \cite{Saff-Totik-1997},
and there are also probabilistic formulations (diffusion limited aggregation, DLA), see \cite{Levine-Peres-2010}. 
The terminology of smashing goes back to \cite{Diaconis-Fulton-1990}.

The relation (\ref{rhot}) turns in the case of $\rho_t$ into 
\begin{equation}\label{sigmat}
\rho_t=
\begin{cases}
2-t & \text{in\, } \D(-1,\sqrt{2})\cap \D(1,\sqrt{2}),\\
1 &\text{in\, }D(t),
\end{cases}
\end{equation}
for a certain doubly connected domain $D(t)$ which surrounds and attaches to the overlapping disks.
Obviously $\rho_t\ne g_t$  for $0<t<1$, but a less trivial fact is that not even the outer
boundaries of $D(t)$ and $\Omega(t)$ are the same.
One argument to show this runs as follows (just a sketch).

Suppose these two outer boundaries were the same. Then ${\rm supp\,}\rho_t={\rm supp \,} g_t$,
and $\rho_t$ and $g_t$ would differ only inside  $\D(-1,\sqrt{2})\cap \D(1,\sqrt{2})$. In that region
we have $\rho_t=2-t$, while by (\ref{rhot})
$$
g_t=
\begin{cases}
2& \text{in\, } \big(\D(-1,\sqrt{2})\cap \D(1,\sqrt{2})\big)\setminus \Omega(t),\\
1 &\text{in\, }\Omega(t)\cap \big(\D(-1,\sqrt{2})\cap \D(1,\sqrt{2})\big).
\end{cases}
$$

Subtracting one unit it would follow that the density $1-t$ (originating from $\rho_t$) would be ``gravi-equivalent''
to the density $\chi_{\rm hole}$ (referring here to the time dependent ``hole'' in 
equation (\ref{rhot})). Observe that this hole is compactly contained in $\D(-1,\sqrt{2})\cap \D(1,\sqrt{2})$, and that 
this latter domain has corners, hence has non-smooth boundary.
It is easy to see that the above, hypothesized, situation contradicts existing theory \cite{Friedman-1982, Petrosyan-Shahgholian-Uraltseva-2012}
for regularity of free boundaries in obstacle type problems: the outer boundary  $\partial\big(\D(-1,\sqrt{2})\cap \D(1,\sqrt{2})\big)$
would have to be analytic.
Obviously it is not analytic, so the assumption that the outer boundaries of $D(t)$
and $\Omega(t)$ are the same must be wrong. The outer boundary of $D(t)$ is still analytic, but probably not
algebraic, in contrast to the boundary of $\Omega(t)$.
  

\subsection{Spherical metric}\label{sec:spherical} 

It is interesting to recast the entire situation using the spherical area measure in place of the Euclidean one. 
It is well-known that the class of classical quadrature domains then remains the same, but strengths and locations 
of quadrature nodes are affected. 

Recall first that classical (finite point) quadrature domains are characterized by the Schwarz function of the boundary
being meromorphic inside the domain \cite{Aharonov-Shapiro-1976, Davis-1974}. The heart of the computation is
$$
\int_\Omega h(z)d\bar{z}dz=\int_{\partial \Omega}h(z)\bar{z}dz=\int_{\partial \Omega}h(z)S(z)dz
$$
$$
=2\pi\I\sum_{{\rm poles\,in\,}\Omega}\res h(z)S(z)dz.
$$
The poles are those of $S(z)$, and $h(z)$ is assumed to be holomorphic in $\Omega$.
In the case of simple poles the sum above will simply becomes that sum of the residues of $S(z)$
times the values of $h(z)$.

For the spherical  metric the corresponding computation becomes
$$
\int_\Omega h(z)\frac{d\bar{z}dz}{(1+|z|^2)^2}
=\int_{\partial\Omega}  \frac{h(z)\bar{z}dz}{1+z\bar{z}}
=\int_{\partial\Omega} \frac{h(z)S(z)dz}{1+zS(z)}
$$
$$
=2\pi\I\sum_{{\rm poles\,in\,}\Omega}\res \frac{h(z)S(z)dz}{1+zS(z)}.
$$
It is now the poles of $S(z)/(1+zS(z))$ that count, and typically these will come from just the zeros of $1+zS(z)$.
See \cite{Gustafsson-2023} for detailed discussions and computations leading to fairly explicit formulas
for two point quadrature domains as in the present paper. The hippopede $\Omega(1)$ mentioned in 
Section~\ref{sec:level lines} is indeed the inversion of a classical quadric, the ellipse. 

Returning to the original problem with the two overlapping disks of radii $\sqrt{2}$, the total area of these is $4\pi$.
Even though it is not really relevant, we note that the unit sphere embedded in three space also has area $4\pi$.
After a stereographic projection from the sphere one obtains the area form
$$
\frac{4dxdy}{(1+(x^2+y^2))^2}=\frac{2\I dzd\bar{z}}{(1+|z|^2)^2}
$$
in the complex plane. The spherical area of the disk $\D(0,r)$ then becomes
$$
2\pi \int_0^r\frac{4sds}{(1+s^2)^2}
=\frac{4\pi r^2}{1+r^2}.
$$
Thus, with $r=1$ we get area $2\pi$, and for $r=\sqrt{2}$ 
it becomes $8\pi/3$. If we move the latter disk to become $\D(1,\sqrt{2})$  the area
decreases to $2\pi$, exactly half of the area of the full sphere. This can be proved 
by moving the center to the origin by a rigid M\"obius transformation, see details below.
 
It should be noted that the radius $r$ and the center $a$ for a disk when writing  $\D(a,r)$
refer only to coordinate values. The disk remains a disk also in terms of distances 
as defined via the intrinsic spherical geometry,
but the geodesic radius and corresponding center are then different. 
Our main observation for this special case can be summarized as follows. 

\begin{proposition}
For the two overlapping disks $\D(-1,\sqrt{2})$ and $\D(1,\sqrt{2})$, 
having maximum radius as for positive semi-definiteness of $1-E_g (z,w)$, 
each disk reaches out, when considered as a disk with respect to the spherical metric, exactly to the
geodesic center of the other disk.

In addition, the total spherical area of the disks, counting multiplicities, equals the area of the full sphere,
which therefore therefore after smashing (partial balayage to density one) becomes maximally filled.

After a rotation of the Riemann sphere the geometry of the two disks is exactly that of two orthogonally crossing half-planes.
\end{proposition}

\begin{proof}
We may start with the observation that the two circles intersect at $\pm \I$ and that these points are antipodal
points on the Riemann sphere, i.e. are related as $z$ and $-1/\bar{z}$. Such a pair of points can be mapped by a rigid 
M\"obius transformation (rotation of the sphere) to $w=0$ and $w=\infty$. Recall that a general rotation of the Riemann sphere
is described by a M\"obius transformation of the form 
$$
w=\frac{az+b}{-\bar{b}z+\bar{a}}.
$$

We first choose ($a=e^{\I\pi /4}$, $b=\I e^{\I\pi /4}$, for example)
$$
w=\frac{z+\I}{z-\I},
$$
which maps the intersection points $z=\pm \I$ of the circles  to $w=0$ and $\infty$.
The centers $z=\pm 1$ are mapped to $w=\pm \I$ (coupled signs), and
the rightmost point $z=\sqrt{2}-1$ on the circle $\partial\D(-1,\sqrt{2})$ and the leftmost point 
$z=1-\sqrt{2}$ on $\partial\D(1,\sqrt{2})$ are mapped, respectively, to
$$
w=-\frac{1}{\sqrt{2}}(1-\I) \quad \text{and} \quad w=-\frac{1}{\sqrt{2}}(1+\I).
$$
It now follows  that the two circles map to the straight lines $\re w=\pm\im w$.
The disk $\D(1,\sqrt{2})$ then corresponds to the half-plane $\re w+ \im w>0$
and $\D(-1,\sqrt{2})$ to $\re w-\im w>0$, the overlapping part hence becoming a quarter plane.
Since the transformation was rigid it follows that 
the total spherical area (with multiplicities) of the two disks equals the full area of the sphere.

We next claim that the Euclidean disk $\D(1,\sqrt{2})$ is simultaneously a geodesic disk with center $z=\sqrt{2}-1$. 
To show this we use another rigid M\"obius transformation,
$$
u=\frac{z-(\sqrt{2}-1)}{(\sqrt{2}-1)z+1},
$$ 
which moves the presumed center $z=\sqrt{2}-1$ to the origin $u=0$. 
The point $z=-1-\sqrt{2}$ corresponds to  $u=\infty$, and in the $u$-plane these two points are mirror points
with respect to a circle centered at $u=0$. Hence the corresponding points $z=-1\pm\sqrt{2}$ in the $z$-planes are also 
mirror points with respect to a circle (or possibly a straight line). And this can indeed be seen to be
the circle $\partial\D(1,\sqrt{2})$.

The conclusion is that $z=\sqrt{2}-1$ is the geodesic center of the disk $\D(1,\sqrt{2})$, and this point lies
on the boundary of the disk $\D(-1,\sqrt{2})$. Hence we may say that what stops the penetration is that one circle reaches the
geodesic center of the other circle. 
\end{proof}

A final remark is that the word geodesic center can alternatively be replaced by harmonic center 
(quadrature node for a one point quadrature identity), because
these are the same for surfaces of constant curvature, see \cite{Gustafsson-Roos-2018}.


\bibliographystyle{plain}
\bibliography{bibliography_gbjorn}

\end{document}